\documentclass[twoside,11pt]{article}

\usepackage{blindtext}

\usepackage[preprint]{jmlr2e}

\usepackage[utf8]{inputenc}
\usepackage[T1]{fontenc}
\usepackage{lmodern}
\usepackage{amsmath}
\usepackage{paralist}
\usepackage{float}
\usepackage{accents}
\usepackage{color}
\usepackage{bm}
\usepackage{lscape}
\usepackage{subfig}
\usepackage{longtable}
\usepackage{multibib}
\usepackage{booktabs}
\newcites{suppl}{References} 
\usepackage{tikz}
\usetikzlibrary{positioning}
\usepackage{dsfont}
\usepackage{pgfplotstable}
\usepackage{hyperref}
\usepackage{lastpage}
\usepackage{multirow}
\usepackage{setspace}

\newcommand{\eps}{{\varepsilon}}
\renewcommand{\phi}{\varphi}
\newcommand{\R}{\mathbb{R}}
\newcommand{\Z}{\mathbb{Z}}
\newcommand{\N}{\mathbb{N}}
\newcommand{\pr}{\mathbb{P}}       
\newcommand{\ex}{\mathbb{E}}       
\newcommand{\var}{\textnormal{Var}} 
\newcommand{\cov}{\textnormal{Cov}}

\newcommand{\Nc}{\mathcal{N}}

\newcommand{\Fc}{\mathcal{F}}
\newcommand{\Dc}{\mathcal{D}}
\newcommand{\Ac}{\mathcal{A}}

\newcommand{\Oc}{\mathcal{O}}
\newcommand{\Ub}{\mathbf{U}}
\newcommand{\Zb}{\mathbf{Z}}
\newcommand{\Mb}{\mathbf{M}}
\newcommand{\Rb}{\mathbf{R}}
\newcommand{\Tb}{\mathbb{T}}
\newcommand{\diff}{{\,\mathrm{d}}}

\newcommand{\convw}{\rightsquigarrow}

\newcommand{\id}{\mathds{1}}

\newtheorem{assumption}[theorem]{Assumption}

\jmlrheading{24}{2025}{1-\pageref{LastPage}}{08/25}{-}{}{Florian Heinrichs}

\ShortHeadings{Self-Normalization for Locally Stationary Time Series}{Heinrichs}
\firstpageno{1}

\begin{document}
	
\title{Self-Normalization for CUSUM-based Change Detection in Locally Stationary Time Series}

\author{\name Florian Heinrichs \email f.heinrichs@fh-aachen.de \\
	\addr FH Aachen\\
	Heinrich-Mußmann-Straße 1\\
	52428 Jülich, Germany}

\editor{-}

\maketitle

\begin{abstract}%
	A new bivariate partial sum process for locally stationary time series is introduced and its weak convergence to a Brownian sheet is established. This construction enables the development of a novel self-normalized CUSUM test statistic for detecting changes in the mean of a locally stationary time series. For stationary data, self-normalization relies on the factorization of a constant long-run variance and a stochastic factor. In this case, the CUSUM statistic can be divided by another statistic proportional to the long-run variance, so that the latter cancels, avoiding estimation of the long-run variance. Under local stationarity, the partial sum process converges to $\int_0^t \sigma(x) \diff B_x$ and no such factorization is possible. To overcome this obstacle, a bivariate partial-sum process is introduced, allowing the construction of self-normalized test statistics under local stationarity. Weak convergence of the process is proven, and it is shown that the resulting self-normalized tests attain asymptotic level $\alpha$ under the null hypothesis of no change, while being consistent against abrupt, gradual, and multiple changes under mild assumptions. Simulation studies show that the proposed tests have accurate size and substantially improved finite-sample power relative to existing approaches. Two data examples illustrate practical performance. 
\end{abstract}


\maketitle 

\section{Introduction} \label{sec:intro}

In diverse fields, such as economics, climatology, engineering, hydrology or genomics, time-dependent observations are analyzed. As the behavior of such time series can vary over time, the study of changes, referred to as \textit{change point analysis}, has gained considerable interest in the last few decades. Most of the recent results are well documented in the review papers by \cite{aue2013, jandhyala2013, woodall2014, sharma2016, chakraborti2019, truong2020} and, more recently, \cite{cho2024}. In the simplest case, one is interested in identifying structural changes in a sequence of means $(\mu_i)_{i=1, \dots, n}$ of a possibly non-stationary time series $(X_i)_{i=1, \dots, n}$. The additive model, 
\[ X_i = \mu_i + \eps_i \]
for $i = 1, \dots, n$, allows us to decompose the time series into a deterministic mean and the associated random errors. Often the mean $\mu_i = \mu(i/n)$ is assumed to be a piecewise constant function $\mu: [0, 1] \to \R$, and the errors to be stationary. A major portion of the literature on change point detection focuses on functions with at most one change point \citep[see, e.\,g.,][among others]{priestley1969, wolfe1984, horvath1999}, but more recently the problem of detecting multiple changes has found notable attention \citep[see, e.\,g.,][among others]{frick2014, fryzlewicz2018, baranowski2019}. While in some applications, the assumption of a piecewise constant mean function is reasonable \citep[see, e.\,g.,][among others]{aston2012, hotz2013, cho2015, kirch2015}, in many settings it is unrealistic. Most (physical) processes, if observed often and long enough, exhibit smooth changes. Examples include climate data \citep{karl1995,collins2000}, financial data \citep{vogt2015} and medical data \citep{gao2019}. 

In applications, where the distribution of an observed time series is expected to vary over time, the rigid framework of stationarity and a (piecewise) constant mean is too restrictive. A more flexible framework is provided by the concept of \textit{local stationarity}. Whereas different notions of local stationarity exist in the literature, the underlying idea is always the same, that short excerpts of the time series seem stationary \citep{dahlhaus1996, zhou2009, birr2017, vogt2012}. Recent research has increasingly focused on the detection of gradual changes in locally stationary time series \citep{vogt2015, dette2019, bucher2020, bucher2021}, and subsequently on the detection of abrupt changes \citep{wu2024}. 

One of the most important approaches to change point detection is the CUSUM statistic, dating back to a seminal work by \cite{page1954}. The idea is essentially that the partial sum process $S_n(t) = \tfrac{1}{n} \sum_{i=1}^{\lfloor nt \rfloor} X_i$ of a stationary time series converges weakly to a Gaussian process. More specifically, for a stationary time series, under mild assumptions, 
\begin{equation}\label{eq:fclt_stat}
	\big\{\sqrt{n}\big( S_n(t) - \ex[S_n(t)] \big)\big\}_{t\in[0, 1]} \convw \sigma B,
\end{equation}
where $\sigma^2$ denotes the \textit{long-run variance} of the time series $(X_i)_{i\in \Z}$ and $B = \{B(t)\}_{t\in[0, 1]}$ denotes a standard Brownian motion. Now, if the mean function $\mu$ is constant, the CUSUM statistic 
\[ \sup_{t \in [0, 1]} \sqrt{n} |S_n(t) - t S_n(1)| \]
converges weakly to $\sigma \cdot \sup_{t \in [0, 1]} | B(t) - t B(1)|$, and diverges to infinity, if $\mu$ is not constant. To derive a statistical test, the unknown long-run variance $\sigma^2$ needs to be estimated. In order to avoid a direct estimation of $\sigma^2$, ratio statistics and self-normalization have been introduced \citep{horvath2008, shao2010}. Since these early works, self-normalization has been extended to various settings \citep[see][for a recent review]{shao2015}. The fundamental idea is to divide the CUSUM statistic by another statistic, which is (asymptotically) proportional to $\sigma$. The long-run variance cancels and the limiting distribution is pivotal.
In a seminal work, \cite{shao2010b} consider a ratio of a (squared) CUSUM-type numerator and a self-normalizer $V_n(k)$ built from partial sums. A key property of their construction is that, under a mean change, the self-normalizer diverges at the same order as the numerator except near the true change point. Extensions to multiple change points in the stationary setting include \cite{zhang2018}, who adapt the self-normalizer to multiple changes, and \cite{zhao2022}, who further generalize the approach to changes in general functionals of the marginal distribution, including multivariate settings. 
More recently, \cite{cheng2024} proposed a locally self-normalized framework for multiple change point testing based on windowed normalizers of width $d$, taking a supremum over $d\in[\eps n,n]$. All of the aforementioned literature consider the alternative of a fixed number of change points, with distances that grow proportionally to $n$, and are not designed for gradual changes in the mean function.

Under local stationarity, when the properties of the time series can vary over time, this approach generally fails. The functional central limit theorem, corresponding to \eqref{eq:fclt_stat}, is given by
\begin{equation*} 
	\big\{\sqrt{n}\big(S_n(t) - \ex[S_n(t)]\big)\big\}_{t\in[0, 1]} \convw \bigg\{\int_0^t\sigma(x) \diff B_x \bigg\}_{t\in [0, 1]},
\end{equation*}
where $\sigma^2(t)$ denotes the possibly time-varying long-run variance. In this case, the limiting distribution does not factorize into a product of $\sigma$ and a term that does not depend on $\sigma$, which complicates self-normalization. Crucially, the limit of $S_n(t)$ only depends on the long-run variance $\sigma(x)$ on the interval $[0, t]$. A ``universal'' sequence of random variables that factors out $\sigma$ necessarily depends on its values on the whole interval $[0, 1]$, which contradicts the previous observation. In general, there is no universal sequence of random variables $Z_n$, so that, for all $t \in[0, 1]$, 
\[ \frac{\sqrt{n}\big(S_n(t) - \ex[S_n(t)]\big)}{Z_n} \]
converges to some limit, that does not depend on $\sigma$.

Different solutions exist to mitigate the intricate limiting distribution. For example, \cite{zhao2013} and \cite{rho2015} consider \textit{modulated time series} following the model $X_i = \mu(i/n) + \sigma(i/n) \eps_i$, for deterministic functions $\mu$ and $\sigma$, and an associated stationary error process. While this model, for (Lipschitz) continuous functions $\mu$ and $\sigma$, yields locally stationary processes, it restricts the non-stationarity of $(X_i)_{i\in\N}$ to non-stationarity in the mean and covariance. In both works a bootstrap procedure, the \textit{wild bootstrap}, is combined with a self-normalized statistic for stationary time series. 
\cite{heinrichs2021} consider a general class of locally stationary time series and propose a self-normalized test statistic for the \textit{relevant} null hypothesis $\int_0^1 (\mu(t) - g(\mu))^2 \diff t \le \Delta$, for some pre-specified threshold $\Delta \ge 0$ and a functional $g(\mu)$. For $\Delta = 0$ and $g(\mu) = \int_0^1 \mu(t) \diff t$, this null hypothesis is equivalent to $\mu$ being constant. Their approach relies on permuting the observations to control the data proportion used for local linear estimation of $\mu$, while simultaneously guaranteeing that data from the full interval $[0, 1]$ is used. The test statistic is based on the $L^2$-norm of the estimator of $\mu$, which has two disadvantages. First, the $L^2$-norm averages deviations over time, so that the test is expected to be insensitive to local, short changes. Moreover, it requires the selection of a kernel function $K$ and a bandwidth $h_n$. For stationary error processes, it has been observed that CUSUM methods, based on the supremum norm, are generally more powerful compared to approaches based on a local estimation of $\mu$ \citep[see, e.\,g.,][]{heinrichs2023}.
In the following, we build on the same permutation idea, but use it in a fundamentally different way. We introduce a bivariate partial sum process $S_n(t, s)$ and derive its weak convergence to a Brownian sheet, which enables a self-normalized CUSUM-type test with a pivotal limit under local stationarity. In contrast to \cite{heinrichs2021}, whose kernel-based estimation requires twice continuous differentiability of $\mu$, the proposed test is consistent against piecewise Lipschitz continuous alternatives and thus allows for abrupt changes. Moreover, the method is developed under weaker regularity conditions on the error process, as outlined in Remark \ref{rem:assumptions}. 

In the following, we are interested in the null hypothesis
\begin{equation} \label{eq:hypothesis}
	H_0: \mu(x) = \mu(0) ~ \forall x\in[0, 1]  \quad \mathrm{vs.} \quad H_1: \exists x\in[0, 1]: \mu(x) \neq \mu(0).
\end{equation}
With this formulation, the alternative covers multiple change points, if $\mu$ is piecewise constant, gradual changes for smooth $\mu$ and combinations thereof, for piecewise continuous functions.
Fundamentally, the proposed test is based on the fact that 
\[ \sup_{t \in [0, 1]}\bigg|\int_0^t \sigma(x) \diff B_x\bigg| \stackrel{\Dc}{=} \bigg(\int_0^1 \sigma^2(t)\diff t\bigg)^{1/2} \sup_{t\in[0, 1]} |B(t)|. \]
This factorization can indeed be used to derive a test for the hypotheses
\begin{equation}\label{eq:simple_hyp}
	\tilde{H}_0: \mu(x) = 0 ~ \forall x\in[0, 1]  \quad \mathrm{vs.} \quad \tilde{H}_1: \exists x\in[0, 1]: \mu(x) \neq 0.
\end{equation}
The construction of a test for the general hypotheses from \eqref{eq:hypothesis} is technically more complex, and relies on the main theoretical contribution, a functional central limit theorem for a double-indexed partial sum process that converges to a Brownian sheet integral under mild assumptions and appears to be of independent interest.
The developed tests are based on this process, which is presented, jointly with mathematical preliminaries in Section \ref{sec:process}. 
Subsequently, tests for the hypotheses in \eqref{eq:hypothesis} and \eqref{eq:simple_hyp} are developed in Section \ref{sec:change_detection}. Section \ref{sec:empirical} contains an extensive simulation study and applications to real data. Section \ref{sec:conclusion} concludes the paper, while proofs of the main results are deferred to Section \ref{sec:proof}. 

Throughout this paper, $\stackrel{\Dc}{=}$ denotes equality of distributions, the symbol $\convw$ denotes weak convergence, and all convergences are for $n\to\infty$, if not mentioned otherwise.

\section{Bivariate Partial Sum Process}\label{sec:process}

In the following, we consider the additive model 
\begin{equation}\label{eq:additive_model}
	X_{i, n} = \mu_{i, n} + \eps_{i, n},
\end{equation}
for $1 \le i \le n$ and $n\in \N$, where $\mu$ denotes a deterministic mean function and $\eps$ a triangular array of centered errors. The mean function is piecewise Lipschitz continuous, as specified in Assumption \ref{assump:mu}, and the error process is locally stationary, as described by Assumption \ref{assump:error}.
We are interested in testing for gradual and abrupt changes, and consider the hypotheses
\[
H_0: \mu_{i, n} = \mu_{1, n} ~\forall i=1, \dots, n \quad \mathrm{vs.} \quad H_1: \exists i \in \{1, \dots, n\}: \mu_{i, n} \neq \mu_{1, n}.
\]
By considering ``rescaled time'' $\frac{i}{n}$, we may rewrite $\mu_{i, n} = \mu\big(\frac{i}{n}\big)$ and equivalently consider the hypotheses in \eqref{eq:hypothesis}.

Let $(b_n)_{n\in\N}$ be a sequence with $b_n\to \infty$ and $b_n = o(n)$, as $n\to\infty$. Further, define the sequence $\ell_n = \lfloor n/b_n\rfloor$. In the following, we split the $n$ observations $X_1, X_2, \dots, X_n$ into $\ell_n$ blocks of length $b_n$, and a remainder of length $n-b_n \ell_n$. Based on these blocks, we define a partial sum process in two arguments $s$ and $t$, that specify which observations are used to calculate the partial sums. More specifically, recall the permutation $\pi$ on the integers $\{1, \dots, n\}$, as introduced by \cite{heinrichs2021}, where $k$ is mapped onto $\pi_k$ with
\[ \pi_k = \left\{ \begin{array}{ll}
	(k-1 \mod \ell_n)b_n + \lceil k/\ell_n\rceil, & \mathrm{if}~k\le\ell_n b_n, \\
	k, & \mathrm{if}~k > \ell_n b_n.
\end{array} \right. \]
The permutation $\pi$ maps the first $\ell_n$ integers onto the first element of the $\ell_n$ blocks, 
so that 
\[ (1, 2, \dots, \ell_n) \mapsto (1, b_n+1, 2b_n + 1, \dots, (\ell_n-1)b_n +1). \]
The next $\ell_n$ integers are mapped onto the second element of each block
\[ (\ell_n + 1, \ell_n + 2, \dots, 2 \ell_n) \mapsto (2, b_n+2, 2b_n + 2, \dots, (\ell_n-1)b_n +2) \]
and so on. We define the partial sum process $S_n = \{S_n(t, s)\}_{t, s\in[0, 1]}$ in terms of 
\begin{equation*} 
	S_n(t, s) = \frac{1}{n} \sum_{i=1}^{n} X_{\pi_i, n} \id( i \le \lfloor tn \rfloor, \pi_i \le \lfloor sn \rfloor),
\end{equation*}
where the parameter $t$ controls the proportion of elements from the blocks $B_k = \{(k-1)b_n + 1, (k-1) b_n + 2 \dots, kb_n \}$, for $k \in \{1, \dots, \ell_n\}$, and $s$ controls the proportion of elements from the entire sample, that is used for the calculation of $S_n$. A graphical illustration of the indices of $S_n(t, s)$ can be found in Figure \ref{fig:process}. 

\begin{figure}
	\includegraphics[width=\linewidth]{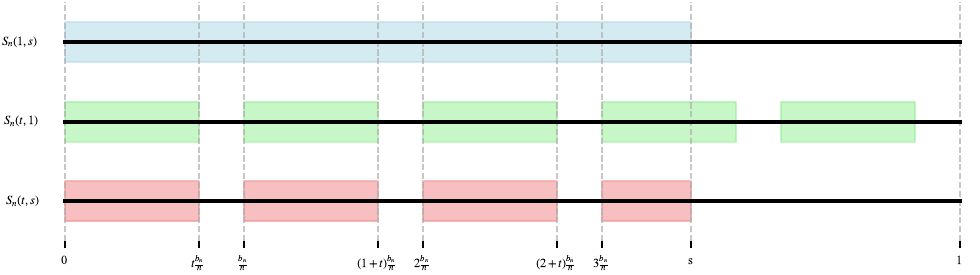}
	\caption{Visualization of indices of the bivariate partial sum process}
	\label{fig:process}
\end{figure}

For $t=1$, we obtain the ordinary partial sum process $S_n(1, s) =  \frac{1}{n} \sum_{i=1}^{\lfloor sn \rfloor} X_{i,n}$, and for $s=1$ we obtain a partial sum process $S_n(t, 1) = \frac{1}{n} \sum_{i=1}^{\lfloor tn \rfloor} X_{\pi_i,n}$, that uniformly covers the full interval. 

In the following, we work with the framework of local stationarity, as proposed by \cite {zhou2009}, presented below. Let $\eta = (\eta_i)_{i\in\Z}$ be a sequence of independent and identically distributed random variables, and let $\eta^* = (\eta_i^*)_{i\in\Z}$ be an independent copy of $\eta$. Further, define $\Fc_i = (\eta_k)_{k \le i}$ and $\Fc_i^* = (\dots, \eta_{-2}, \eta_{-1}, \eta_0^*, \eta_1, \dots, \eta_i)$. Let $H:[0, 1] \times \R^\infty \to \R$ denote a (possibly non-linear) map, such that $H(t, \Fc_i)$ is measurable for all $t\in [0, 1], i\in\N$.

The \textit{physical dependence measure} of a map $H$ with $\sup_{t\in[0, 1]} \ex[H^2(t, \Fc_i)] < \infty$ is defined by 
\[ \delta(H, i) = \sup_{t\in[0, 1]} \ex\big[\big(H(t, \Fc_i) - H(t, \Fc_i^*)\big)^2\big]^{1/2}. \]
The quantity $\delta(H, i)$ measures the strength of the serial dependence of $H(t, \Fc_i)$ and plays a similar role as mixing coefficients. Further, a triangular array $\{(\eps_{i, n})_{1\le i\le n}\}_{n\in\N}$ is called \textit{locally stationary}, if there exists some map $H$, which is continuous in its first argument, such that $\eps_{i, n} = H(i/n, \Fc_i)$, for all $i = 1, \dots, n$ and $n\in\N$. The map $H$ is \textit{Lipschitz continuous with respect to the $L^2$-norm,} if
\[ \sup_{0 \le s < t \le 1} \ex \big[ \big( H(t, \Fc_i) - H(s, \Fc_i) \big)^2\big]^{1/2} / |t - s| < \infty. \]

\begin{assumption} \label{assump:error}
	Let the triangular array $\{(\eps_{i, n})_{1\le i\le n}\}_{n\in\N}$ in \eqref{eq:additive_model} be centered and locally stationary with map $H$, such that the following conditions are satisfied:
	\begin{enumerate}
		\item $\Theta_m = \sum_{i=m}^\infty \delta(H, i)$ vanishes as $m \to\infty$.
		\item The map $H$ is Lipschitz continuous with respect to the $L^2$-norm, and moments of order $4$ are uniformly bounded, i.\,e., $\sup_{t\in[0, 1]} \ex[H^4(t, \Fc_0)] < \infty$.
		\item The (local) long-run variance of $H$, defined as
		\begin{equation*}
			\sigma^2(t) = \sum_{i=-\infty}^\infty \cov\big(H(t, \Fc_i), H(t, \Fc_0)\big),
		\end{equation*}
		for $t\in[0, 1]$, exists and is Lipschitz continuous.
	\end{enumerate}
\end{assumption}

\begin{assumption} \label{assump:sequences}
	The sequence $(b_n)_{n\in\N}$ diverges to $\infty$ such that  $\lim_{n\to\infty} \frac{b_n^2}{n} = 0$. Moreover, a sequence $(m_n)_{n\in\N}$ exists, such that $\lim_{n\to\infty} \frac{m_n^2}{b_n} = 0, \lim_{n\to\infty} \frac{b_n^2 m_n^2}{n} = 0$ and $\lim_{n\to\infty} \sqrt{n} \Theta_{m_n} = 0$.	
\end{assumption}

\begin{assumption} \label{assump:mu}
	The function $\mu$ is piecewise Lipschitz continuous on $[0, 1]$.
\end{assumption}

\begin{remark} \label{rem:assumptions}
	The assumptions are rather mild. 
	\begin{enumerate}
		\item Assumption \ref{assump:error} is weaker, than usual regularity conditions for non-stationary error processes \citep[see, e.\,g.,][]{bucher2021, heinrichs2021}. In contrast to the literature, $\delta(H, i)$ is defined in terms of the $L^2$-norm instead of the $L^4$-norm, and it only needs to vanish sufficiently fast, rather than exponentially. Furthermore, $H$ must be Lipschitz continuous with respect to the $L^2$- rather than the $L^4$-norm, and fourth-order moments must be uniformly bounded instead of eighth-order moments. Finally, while it is often assumed that $\sigma^2(t) > 0$ for all $t\in[0, 1]$, this assumption is relaxed to  allow $\sigma^2(t) = 0$. In the degenerate case $\sigma \equiv 0$, Theorem \ref{thm:functional_clt} is trivial. Part (3) of Assumption \ref{assump:error} follows from (1), if we additionally assume that $\sum_{m=1}^\infty \Theta_m < \infty$.
		\item When proving weak convergence of $G_n$, we use the big-blocks-small-blocks method, where the big blocks are independent and the small blocks asymptotically negligible. Due to the block structure of $S_n$, the length of consecutive big and small blocks will naturally be $b_n$. With the small block length $m_n$, big blocks will have length $b_n - m_n$. Asymptotic negligibility of the small blocks requires sufficient weak dependence. The error term associated with the small blocks is of order $\sqrt{n}\Theta_{m_n}$ and is assumed to vanish. 
		Error terms of order  $\tfrac{b_n^2}{n}$ arise in multiple locations and are due to Lipschitz continuity of $\sigma^2$. More specifically, when approximating $\gamma_h(t) := \cov\big(H(t, \Fc_h), H(t, \Fc_0)\big)$ by $\gamma_h(j/\ell_n)$, for $t \in [\tfrac{j-1}{\ell_n}, \tfrac{j+1}{\ell_n}]$, the error is of order $\Oc(b_n / n)$. Summing over $b_n$ such terms yields $\Oc(b_n^2/n)$. 
		The other leading error terms stem from the chaining arguments in the proof of Lemma \ref{lem:equicont}. 
		
		Assumption \ref{assump:sequences} states, that the error terms vanish. It is satisfied, for example, whenever $\delta(H, i) \le \gamma^i$, for some $\gamma \in (0, 1)$. In this case, $\Theta_{m_n} = \Oc(\gamma^{m_n})$, and the assumption is satisfied with $b_n = n^{1/2 - \eps}$ and $m_n = n^{\eps/2}$, for $\eps \in (0, \tfrac{1}{4})$.
		
		If $\delta(H, i)$ vanishes algebraically, i.\,e., $\delta(H, i) = \Oc(i^{-p})$, for some $p > 4$, $\Theta_m$ is of order $\Oc(m^{-p+1})$ by the integral test for convergence of the series. With $m_n = n^\beta$, for $\beta = \tfrac{1}{6(p-1)}+\tfrac{1}{9}$ and $b_n = n^{1/3}$, the term $\sqrt{n}\Theta_{m_n} = n^{1/3 - (p-1)/9}$ vanishes. Similarly, $m_n^2/b_n = b_n^2m_n^2/n = n^{1/(3p-3)-1/9}$ vanish too, so that the assumption is satisfied. 
				
		\item Assumption \ref{assump:mu} is substantially weaker compared to conditions from the literature, where $\mu$ is often assumed to be twice differentiable with Lipschitz continuous second derivative \citep[see, e.\,g.,][]{bucher2021, heinrichs2021}. Here, we only assume that it is piecewise Lipschitz continuous. The condition is required to derive consistency of the tests in Section \ref{sec:change_detection}.
	\end{enumerate}
\end{remark}


\begin{theorem} \label{thm:functional_clt}
	Let Assumptions \ref{assump:error} and \ref{assump:sequences} be satisfied. Then, the centered partial sum process $G_n = \{G_n(t,s)\}_{t, s\in[0, 1]}$, with 
	\[ G_n(t, s) = \sqrt{n} \big( S_n(t, s) - \ex\big[S_n(t, s)\big] \big), \]
	converges weakly to $\{G(t, s)\}_{t, s\in[0, 1]}$, where
	\[ G(t, s) = \int_{[0, t]\times [0, s]} \sigma(x) dB(u, x), \]
	for a standard Brownian sheet $B$.
\end{theorem}

As usual in the study of empirical processes, we establish convergence of the finite dimensional distributions and equicontinuity of the process $S_n$. The assertion of Theorem \ref{thm:functional_clt} follows directly with Theorems 1.5.4 and 1.5.7 of \cite{vandervaart1996} from the following two lemmas.

\begin{lemma} \label{lem:fdd_conv}
	Let Assumptions \ref{assump:error} and \ref{assump:sequences} be satisfied. Then 
	\begin{equation} \label{eq:fdd_conv}
		\big( G_n(t_1, s_1), \dots, G_n(t_d, s_d) \big)^T \convw \big( G(t_1, s_1), \dots, G(t_d, s_d) \big)^T
	\end{equation}
	in $\R^d$, for any $t_1, t_2, \dots, t_d, s_1, s_2, \dots s_d \in [0, 1]$ and $d\in\N$.
\end{lemma}

\begin{lemma} \label{lem:equicont}
	Let Assumptions \ref{assump:error} and \ref{assump:sequences}. Then, $G_n$ is stochastically equicontinuous, that is, for any $\eps > 0$, 
	\[ \lim_{\rho \searrow 0} \lim_{n\to\infty} \pr \Big( \sup_{d\big((t_1, s_1), (t_2, s_2)\big) \le \rho} |G_n(t_1, s_1) - G_n(t_2, s_2)| > \eps \Big) = 0. \]
\end{lemma}

The process $G_n$ converges weakly to $G$ for any sequence $(b_n)_{n\in\N}$ that satisfies Assumption \ref{assump:sequences}. While the choice of $b_n$ does not make a difference asymptotically, reasonable values should be selected for finite samples. 
The error term $\sqrt{n}\Theta_{m_n}$ indicates that a suitable choice of the auxiliary truncation sequence $(m_n)_{n\in\N}$ depends on the dependence structure of $\eps$, where $m_n$ can be chosen smaller under weaker dependence. Importantly, $G_n$ does not depend on $m_n$. To obtain a data-agnostic block size $b_n$, we assume that a sufficiently small truncation sequence exists, so that the $m_n$-free error terms dominate the overall error order. Under strong serial dependence, the $m_n$-dependent terms may dominate. 

Careful bookkeeping of the error terms in the proofs of the previous lemmas, yields the dominant $m_n$-free error terms $(b_n/n)^{1/8}, b_n/\sqrt{n}$ and $1/b_n^{1/4}$. For $b_n = n^\alpha$ these terms are $n^{(\alpha - 1)/8}, n^{\alpha - 1/2}$ and $n^{-\alpha/4}$, respectively. Balancing these algebraic terms leads to $\alpha=\tfrac{1}{3}$, which equalizes the first and third terms and gives a joint rate of $\Oc(n^{-1/12})$, so a convenient, data-agnostic block size is $b_n = \lfloor n^{1/3}\rfloor$.

\section{Detecting Change Points and Gradual Changes} \label{sec:change_detection}

In the following, we only consider the non-degenerate case, where $\sigma^2$ is not constantly $0$. Before considering the general hypothesis in \eqref{eq:hypothesis}, we start with the simpler testing problem from \eqref{eq:simple_hyp}. Under $\tilde{H}_0$, it holds that $\ex[S_n(t, s)] = 0$, for all $t, s\in[0, 1]$, so that 
\[ 
	\big\{\sqrt{n} S_n(1, s)\big\}_{s\in [0, 1]} \convw \big\{ G(1, s)\big\}_{s\in [0, 1]}
\]
If furthermore $\ex[S_n(t, 1)] - t \ex[S_n(1, 1)] = o(n^{-1/2})$ uniformly in $t$, under $\tilde{H}_1$, 
\[
	\sqrt{n} \big( S_n(t, 1) - t S_n(1, 1) \big)
	= G_n(t, 1) - t G_n(1, 1) + o(1)
\]
converges weakly to $G(t, 1) - t G(1, 1)$, as a process in $t$. Let $\|\sigma\|=\big(\int_0^1 \sigma^2 (x) \diff x\big)^{1/2}$ denote the $L^2$-norm of $\sigma$, and define $G^{(1)}(t) = G(1, t)$ and $G^{(2)}(t) = G(t, 1) - t G(1, 1)$. 
Then, for any $s, t \in[0,1]$, straightforward calculations yield the covariances
\begin{align*}
	\cov\big(G^{(1)}(s), G^{(1)}(t)\big) & = \int_0^{s \wedge t} \sigma^2(x) \diff x , \notag \\
	\cov\big(G^{(2)}(s), G^{(2)}(t)\big) & = \| \sigma \| \big(s \wedge t - s t\big),\\ 
	\cov\big(G^{(1)}(s), G^{(2)}(t)\big) & = 0 \notag
\end{align*}
so that 
\[
	G^{(1)}(t) \stackrel{\Dc}{=} \int_0^t \sigma(x) \diff B^{(1)}(x), 
	\quad \mathrm{and} \quad
	G^{(2)}(t) \stackrel{\Dc}{=} \| \sigma \| \big(B^{(2)}(t) - t B^{(2)}(1) \big),
\]
for independent Brownian motions $\big\{B^{(1)}(t)\big\}_{t\in [0,1]}, \big\{B^{(2)}(t)\big\}_{t\in [0,1]}$. Moreover, by the Dubins-Schwarz theorem, $G^{(1)}(t) \stackrel{\Dc}{=} B^{(1)}(\int_0^t \sigma^2(x) \diff x)$, so that 
\begin{align}\label{eq:timeshift}
	\sup_{s\in[0, 1]}|G^{(1)}(s)| 
	\stackrel{\Dc}{=} \sup_{s\in[0, 1]}\bigg|B^{(1)}\bigg(\int_0^s \sigma^2(x) \diff x\bigg)\bigg|
	& = \sup_{s\in[0, \|\sigma\|^2]}|B^{(1)}(s)| \\
	& = \sup_{s\in[0, 1]}|B^{(1)}(s \| \sigma \|^2 )| 
	\stackrel{\Dc}{=} \| \sigma \| \sup_{s\in[0, 1]}| B^{(1)}(s)|, \notag
\end{align}
where the second equality follows from non-negativity of $\sigma^2$ and the last equality follows from self-similarity of the Brownian motion. Under $\tilde{H}_0$, 
\begin{align}
	\frac{\sqrt{n} \sup_{s\in[0, 1]}|S_n(1, s)|}{\sqrt{n} \sup_{t\in[0, 1]} |S_n(t, 1) - t S_n(1, 1)|} 
	& \convw 
	\frac{\sup_{s\in[0, 1]}| B^{(1)}(s)|}{\sup_{t\in[0, 1]}	\big|B^{(2)}(t) - t B^{(2)}(1) \big|}, \label{eq:conv_simple_null}
\end{align}
which does not depend on the long-run variance $\sigma^2$. Indeed, the numerator is the maximum of the absolute value of a Brownian motion and the denominator is the maximum of the absolute value of a Brownian bridge, which follows the Kolmogorov distribution. Quantiles of the distribution can be estimated in terms of a Monte Carlo simulation.

Unfortunately, though, the difference $\sqrt{n}\big(\ex[S_n(t, 1)] - t \ex[S_n(1, 1)]\big)$ does not vanish as $n$ approaches $\infty$, due to the block structure of $S_n$. Note that, by Proposition \ref{prop:mu_approx}, 
\begin{equation*}
	\ex[S_n(t, 1)] = \frac{\lfloor \tfrac{n t}{\ell_n}\rfloor}{b_n} \int_0^1 \mu(x) \diff x - \frac{1}{b_n} \int_{(\lfloor nt \rfloor - \lfloor \tfrac{n t}{\ell_n}\rfloor \ell_n)\tfrac{b_n}{n}}^1 \mu(x) \diff x + \Oc\big(\tfrac{b_n}{n}\big),
\end{equation*}
where the lower limit in the second integral can take any value $t_0 \in [0, 1]$ by plugging in $t = \frac{(k+t_0)\ell_n}{n}$, for $k \in \{1, \dots, b_n\}$.  Hence, the contribution of the second integral is of order $\sqrt{n}/b_n$, which grows to $\infty$, since $b_n = o(\sqrt{n})$. 
Instead, define 
\[
	\tilde{S}_n(t, s) := S_n(\lfloor \tfrac{tn}{\ell_n}\rfloor \tfrac{\ell_n}{n}, s)
	\quad \mathrm{and} \quad
	t_n = \tfrac{\lfloor \tfrac{tn}{\ell_n}\rfloor - 1}{\lfloor \tfrac{n}{\ell_n}\rfloor - 1}. 
\]
Clearly, $\lfloor \tfrac{tn}{\ell_n}\rfloor \tfrac{\ell_n}{n}$ and $t_n$ converge to $t$, as $n\to\infty$. By Proposition \ref{prop:mu_approx}, 
\[\sqrt{n}\big(\ex[\tilde{S}_n(t, 1)] - t_n \ex[\tilde{S}_n(1, 1)]\big) = o(1).\] 

Let $q_{1-\alpha}$ denote the $1-\alpha$ quantile of the limiting distribution in \eqref{eq:conv_simple_null}. Then, we can reject $\tilde{H}_0$, whenever
\begin{equation}\label{eq:decision_rule0}
	\frac{\sup_{s\in[0, 1]}|S_n(1, s)|}{\sup_{t\in[0, 1]} |\tilde{S}_n(t, 1) - t_n \tilde{S}_n(1, 1)|} > q_{1-\alpha}.
\end{equation}

\begin{corollary} \label{cor:simple}
	Let Assumptions \ref{assump:error}, \ref{assump:sequences} and \ref{assump:mu} be satisfied, and $\sigma^2$ not constantly $0$. The test defined by the decision rule \eqref{eq:decision_rule0} has asymptotically level $\alpha$ under $\tilde{H}_0$ and is consistent against $\tilde{H}_1$.	
\end{corollary}

We now turn to the more general testing problem from \eqref{eq:hypothesis}. A classic approach is to use the CUSUM statistic $\sup_{s\in[0, 1]}|S_n(1, s) - s S_n(1, 1)|$, which converges, under $H_0$, to
\[ \sup_{s\in[0, 1]}|G^{(1)}(s) - s G^{(1)}(1)|. \]
Though, we cannot use the same time shift from \eqref{eq:timeshift}. If $\sigma^2(x)$ is positive, for all $x\in [0,1]$, the function $M(t) = \int_0^t \sigma^2(x) \diff x$ is invertible. With this notation it holds 
\begin{align*}
	\sup_{s\in[0, 1]}|G^{(1)}(s) - s G^{(1)}(1)| 
	& \stackrel{\Dc}{=} \sup_{s\in[0, 1]}\bigg|B^{(1)}\bigg(\int_0^s \sigma^2(x) \diff x\bigg) - s B^{(1)}\bigg(\int_0^1 \sigma^2(x) \diff x\bigg) \bigg| \\
	& = \sup_{s\in[0, M(1)]}|B^{(1)}(s) - M^{-1}(s) B^{(1)}(1)|,
\end{align*}	
where $M^{-1}(s)$ generally depends on $\sigma^2$, so that we cannot factor out a single constant depending on $\sigma^2$. 

True time $t$ and ''variance'' time $M(t)$ are generally incompatible. The two quantities are only compatible if $\sigma^2$ is constant, hence, $M(t) = t \sigma^2$. For the general testing problem from \eqref{eq:hypothesis}, we restrict our attention to this case. In the following, we construct two (asymptotically) independent processes $V_n$ and $H_n$, such that 
\begin{itemize}
	\item $\ex[V_n(t)] = 0$ for all $t\in[0, 1]$ under $H_0$, 
	\item $\lim_{n\to\infty}|\ex[V_n(t)]| = \infty$ for some $t\in[0, 1]$ under $H_1$,
	\item $\ex[H_n(t)] \approx 0$ under $H_0$ and $H_1$,
	\item $V_n - \ex[V_n] \convw V$ and $H_n \convw H$, for two independent Gaussian processes $V$ and $H$ with (up to constants) the same covariance structure.
\end{itemize}

Due to this latter convergence, we can use a time shift similar to \eqref{eq:timeshift}, to obtain a pivotal limit.
First, fix values $t_0, t_1 \in(0, 1)$ so that  $t_0 < t_1$. Similar to the CUSUM process, for $s \in [0, 1]$, define
\[
	V_n(s) = \sqrt{n} \bigg( \int_0^s \tilde{S}_n(t_0, x) - \frac{x}{s} \tilde{S}_n(t_0, s) \diff x \bigg).
\]
Moreover, let $H_n(s) = \int_0^s \tilde{H}_n(x) - \frac{x}{s} \tilde{H}_n(s) \diff x$, where
\[
	\tilde{H}_n(s) = \sqrt{n} \bigg\{ \tilde{S}_n(t_1, s) - \tilde{S}_n(t_0, s) - \frac{\lfloor \frac{t_1 n}{\ell_n} \rfloor - \lfloor \frac{t_0 n}{\ell_n} \rfloor}{\lfloor \frac{n}{\ell_n} \rfloor - \lfloor \frac{t_0 n}{\ell_n} \rfloor} \big[\tilde{S}_n(1, s) - \tilde{S}_n(t_0, s) \big] \bigg\},
\]
for $s \in [0, 1]$. Finally, let $q_{1- \alpha}$ denote the $1-\alpha$ quantile of $\frac{\sup_{s \in [0, 1]} |B^{(1)}(s)|}{\sup_{s \in [0, 1]} |B^{(2)}(s)|}$, for two independent Brownian motions $B^{(1)}, B^{(2)}$. Then, we reject $H_0$, whenever
\begin{equation}\label{eq:decision_rule}
	\frac{\sup_{s \in [0, 1]} |V_n(s)|}{\sup_{s \in [0, 1]} |H_n(s)|} > \sqrt{\frac{t_0(1-t_0)}{(1-t_1)(t_1 - t_0)}} q_{1-\alpha}.
\end{equation}

By similar arguments as for the decision rule in \eqref{eq:decision_rule0}, the test defined by \eqref{eq:decision_rule} has asymptotically level $\alpha$ and is consistent against alternatives, where $\mu$ is piecewise Lipschitz continuous. In particular, the test is consistent against (multiple) change points and gradual changes.

\begin{corollary} \label{cor:full}
	Let Assumptions \ref{assump:error}, \ref{assump:sequences}  and \ref{assump:mu} be satisfied, and $\sigma^2(x)\equiv \sigma^2 > 0$ be constant. The test defined by the decision rule \eqref{eq:decision_rule} has asymptotically level $\alpha$ under $H_0$ and is consistent against $H_1$.	
\end{corollary}

\begin{remark}
	The construction of $V_n$ and $H_n$ seems overly sophisticated. For a time-varying $\sigma^2$, both statistics converge weakly to $V = \sqrt{t_0} W^{(1)}$ and $H = \sqrt{(1-t_1)(t_1-t_0)/(1 - t_0)} W^{(2)}$, for independent copies of 
	\begin{equation*}
		W(t) = \int_0^t \Big(\frac{t}{2}-z\Big) \sigma(z) \diff B(z),
	\end{equation*} 
	where $B$ denotes a standard Brownian motion. If $W$ was a martingale, by the Dubins-Schwarz theorem, it would have the same distribution as
	$B( \int_0^s (\frac{s}{2}-z)^2 \sigma^2(z) \diff z)$.
	Analogously to \eqref{eq:timeshift}, 
	\begin{equation*}
		\sup_{s \in [0, 1]} \bigg| B\bigg( \int_0^s \Big(\frac{s}{2}-z\Big)^2 \sigma^2(z) \diff z \bigg)\bigg|
		= \sup_{v \in [0, I]}| B(v)|
		= \sup_{v \in [0, 1]}| B(Iv)| 
		\stackrel{\Dc}{=} \sqrt{I} \sup_{v \in [0, 1]}| B(v)|,
	\end{equation*}
	such that
	\begin{equation*}
		\sqrt{\frac{t_0(1 - t_0)}{(1-t_1)(t_1-t_0)}} \frac{\sup_{s \in [0, 1]} | W^{(1)}(s)|}{\sup_{s \in [0, 1]} |W^{(2)}(s)|}
		\stackrel{\Dc}{=} \sqrt{\frac{t_0(1 - t_0)}{(1-t_1)(t_1-t_0)}} \frac{\sup_{v \in [0, 1]}| B^{(1)}(v)|}{\sup_{v \in [0, 1]}|B^{(2)}(v)|},
	\end{equation*}
	which is pivotal again. Now, $W$ is not a martingale and the Dubins-Schwarz theorem cannot be applied. However, the previous considerations explain why the test, defined by \eqref{eq:decision_rule}, seems to work well even for time-varying $\sigma$, as indicated by the finite sample properties in Section \ref{sec:empirical}.
\end{remark}

The corollary is valid for any $0 < t_0 < t_1 < 1$, and asymptotically, the selected values make no difference. However, for finite samples, differences exist and a reasonable choice of $t_0$ and $t_1$ is crucial. By construction, the process $V_n$ depends on $\lfloor t_0 n\rfloor$ observations, hence $t_0$ should be maximal. In contrast, the variance of $H_n$ is proportional to $\tfrac{(1-t_1)(t_1-t_0)}{1-t_0}$. To ensure that the ratio of $V_n$ and $H_n$ is stable, the denominator should be as large as possible. The harmonic mean of $t_0$ and $\tfrac{(1-t_1)(t_1-t_0)}{1-t_0}$ provides a reasonable trade-off. The harmonic mean is given by
\begin{equation*}
	\frac{2}{\frac{1}{t_0} + \frac{1-t_0}{(1-t_1)(t_1-t_0)}} = \frac{2}{\frac{1}{t_0} + \frac{1}{t_1 - t_0} + \frac{1}{1-t_1}},
\end{equation*}
which is maximal whenever the denominator is minimal. By the Cauchy-Schwarz inequality,
\begin{equation*}
	\frac{1}{t_0} + \frac{1}{t_1 - t_0} + \frac{1}{1-t_1}
	= \bigg(\frac{1}{t_0} + \frac{1}{t_1 - t_0} + \frac{1}{1-t_1}\bigg) (t_0 + t_1 - t_0 + 1- t_1) \ge 9,
\end{equation*}
where the minimal value is assumed for $t_0 = \tfrac{1}{3}$ and $t_1 = \tfrac{2}{3}$.

\subsection{Local Alternatives and Monotone Power} \label{sec:local_alternatives}

In the context of self-normalization, two related questions arise. First, whether the test is consistent with respect to local alternatives. And second, whether the test overcomes the ``non-monotone power issue''. The latter describes an effect that occurs in classical self-normalization, where both the numerator and denominator diverge under the alternative, which can lead to a declining power for ``large deviations'' from the null hypothesis. In fact, the decision rule in \eqref{eq:decision_rule} is constructed in such a way that both questions can be answered affirmative.

In the classic ``at most one change'' setting, local alternatives refer to an asymptotically vanishing height of the change, and can be straightforwardly defined. In the present case of a piecewise Lipschitz continuous mean function, we have more degrees of freedom and local alternatives can be defined in various ways. In the following, we consider two representative types of local alternatives. First, let $\tilde{t}\in (0, 1)$ and $(a_n)_{n\in\N}$ be a sequence that vanishes as $n$ grows. Then, we consider the \textit{local abrupt alternative}
\begin{equation*}
	H_1^{\mathrm{(abrupt)}}: \mu(t) = \mu_0 + a_n \id(t \ge \tilde{t}).
\end{equation*}
Second, let $(a_n)_{n\in\N}$ and $(c_n)_{n\in\N}$ be vanishing sequences, $\tilde{t}\in (0, 1)$ and $h:\R\to\R$ a symmetric, non-negative, differentiable function with support $[-1, 1]$ and $\int h(x)\diff x = 1$. Then, we define the \textit{local smooth alternative} as 
\begin{equation*}
	H_1^{\mathrm{(smooth)}}: \mu(t) = \mu_0 + a_n h\big(\tfrac{t - \tilde{t}}{c_n}\big).
\end{equation*}

The test defined by \eqref{eq:decision_rule} is consistent against these local alternatives.

\begin{corollary} \label{cor:local_alternatives}
	Let Assumptions \ref{assump:error} and \ref{assump:sequences} be satisfied, and $\sigma^2(x) \equiv \sigma^2 > 0$ be constant. For 
	\begin{enumerate}
		\item the local abrupt alternative, let $\lim_{n\to\infty} \sqrt{n}a_n = d > 0$, and for
		\item the local smooth alternative, let $\lim_{n\to\infty} \sqrt{n}a_n c_n = d > 0$.
	\end{enumerate} 
	The test defined by decision rule \eqref{eq:decision_rule} 
	\begin{itemize}
		\item is consistent, if $d = \infty$, and
		\item has non-trivial power, if $d < \infty$.
	\end{itemize}
\end{corollary}

In a seminal paper, \cite{lobato2001} proposed a self-normalized statistic to test whether the mean of a stationary time series is zero. \cite{shao2010} adapted this approach to the detection of a single change point. However, empirical studies have shown that the power decreases when the alternative moves further away from the null hypothesis. \cite{shao2010b} explain this non-monotonic power issue by the fact that the test statistic does not take the change point into account under the alternative. Both the numerator and denominator diverge under the alternative. Subsequently, \cite{shao2010b} proposed an adapted version of the test statistic that avoids the non-monotone power issue.

In a similar spirit, the process $H_n$ was constructed to converge weakly to the same limit under both the null hypothesis and the alternative. In contrast, the process $V_n$ was constructed, such that the ratio converges weakly to a pivotal limit under the null hypothesis, and the numerator diverges under the alternative.

\subsection{Estimating the First Point of Change}

In applications, we are usually not only interested in testing for the existence of change points, but also estimating their location. In the context of piecewise Lipschitz continuous mean functions, we can have multiple change points, in fact infinitely many, if $\mu$ is not piecewise constant. In this case, we are interested in the first deviation of $\mu$ from its start value $\mu(0)$, i.\,e., 
\begin{equation*}
	s^* = \inf\{ s\in [0, 1]: \mu(s) \neq \mu(0)\},
\end{equation*}
with the convention $\inf \emptyset = \infty$, in case of no change. The detection of $s^*$ is simple, if $\mu$ has a jump point in $s^*$, and becomes increasingly difficult, the smoother $\mu$ is. To capture the degree of smoothness of $\mu$ at $s^*$, we use an approach similar to \cite{bucher2021}. 
Assume 
that constants $\kappa \ge 0$ and $c_\kappa \neq 0$ exist, such that
\begin{equation} \label{eq:smoothness_change}
	\lim_{s\searrow s^*} \frac{\mu(s) - \mu(0)}{(s - s^*)^\kappa} = c_\kappa.
\end{equation}
Note that $\kappa = 0$, if $\mu$ has a jump point in $s^*$, and $\kappa= 1$, if $\mu$ is differentiable in $s^*$ with non-vanishing derivative. 

Let $(c_n)_{n\in\N}$ be a sequence with $c_n \to\infty$ and $c_n = o(\sqrt{n})$ as $n \to\infty$. Then we can estimate $s^*$ by 
\begin{equation}
	\hat{s}^* = \inf\{ s\in [0, 1]: |V_n(s)| > c_n\}.
\end{equation}

\begin{corollary} \label{cor:estimation}
	Let Assumptions \ref{assump:error}, \ref{assump:sequences} and \ref{assump:mu} be satisfied, and $\sigma^2$ not constantly $0$. Then $\hat{s}^* = s^* + O_\pr\big((\tfrac{c_n}{\sqrt{n}})^{1/(\kappa + 1)}\big)$, if $s^* < \infty$, and $\pr(\hat{s}^* < \infty) = o(1)$, if $s^* = \infty$. In particular, $\hat{s}^*$ is a consistent estimator of $s^*$. 
\end{corollary}

Note that in contrast to Corollary \ref{cor:full}, the long-run variance $\sigma^2$ may vary over time.

\section{Empirical Results} \label{sec:empirical}

We study the finite sample properties of the tests, defined via the decision rules \eqref{eq:decision_rule0} and \eqref{eq:decision_rule}, by means of a large simulation study and illustrate its application in a case study.\footnote{Python implementations of the methods and experiments are available on GitHub: \url{https://github.com/FlorianHeinrichs/cusum_self_normalization}.}

The process $S_n$ depends on the block size $b_n$, and the test statistic in \eqref{eq:decision_rule} depends on the choice of $t_0$ and $t_1$. We selected $b_n = \lfloor n^{1/3}\rfloor, t_0 = \tfrac{1}{3}$ and $t_1 = \tfrac{2}{3}$, as discussed previously. 

For a comparative analysis, we used five alternative approaches. First, we used the tests proposed by \cite{bucher2021}, based on (asymptotic) Gumbel quantiles and quantiles from a Gaussian approximation, referred to as \texttt{R1} and \texttt{R2}, respectively. These tests can only test for constant $\mu$, if the long-run variance $\sigma^2(\cdot)$ is constant. For a time-varying long-run variance, they only test whether the signal to noise ratio $\mu / \sigma$ remains constant. Hence, the global long-run variance estimator 
\begin{equation}\label{eq:lrv_estimator}
	\hat{\sigma}^2 = \frac{1}{1 - 2 m_n + 1} \sum_{i=1}^{n-2m_n+1} \frac{1}{2 m_n}\bigg( \sum_{j=0}^{m_n-1} X_{i+j, n} - \sum_{j=m_n}^{2m_n-1} X_{i+j, n} \bigg)^2,
\end{equation}
for $m_n \sim n^{1/3}$, was used, as defined in eq. (6.3) of \cite{bucher2021}. Further, we used the self-normalization approach by \cite{heinrichs2021}, referred to as \texttt{SN}. The aforementioned tests are based on the local linear estimator, whose bandwidth was tuned with cross validation. Further, theses tests are formulated for ``relevant hypotheses'', which are equivalent to \eqref{eq:hypothesis} for $\Delta = 0$. Moreover, the Bootstrap procedure, from \cite{bucher2020} was used, referred to as \texttt{BT}. Finally, a simple CUSUM-test was used, where the null hypothesis of a constant $\mu$ was rejected, whenever
\[  \sup_{t \in [0, 1]} \frac{1}{\sqrt{n}} \bigg|\sum_{i=1}^{\lfloor tn\rfloor} X_{i, n} - t \sum_{i=1}^n X_{i,n}\bigg| > \hat{\sigma} q_{1-\alpha}^K, \]
where $\hat{\sigma}^2$ denotes the (global) long-run variance estimator from \eqref{eq:lrv_estimator} and $q_{1-\alpha}^K$ is the $(1-\alpha)$-quantile of the Kolmogorov distribution. This latter test is referred to as \texttt{LRV}.

\subsection{Simulation Study}

For the simulation study, we consider the model
\begin{equation*}
	X_{i, n} = \mu(\tfrac{i}{n}) + \sigma(\tfrac{i}{n}) \eps_i,
\end{equation*}
for $i=1, \dots, n$, where $\mu$ denotes the mean function of interest, $\sigma^2(\cdot)$ a (non-) constant variance and $(\eps_i)_{i\in\N}$ an error process.
The following seven different choices of the mean function $\mu$ were considered
\[
\begin{array}{ll}
	\mu_0(x)  = 0, &  \\ 
	\mu_1(x)  = \sin(8 \pi x) + 2 (x - \tfrac{1}{4})^2 \id(x > \tfrac{1}{4}),
	& \mu_4(x)  = \tfrac{1}{2} - \mu_1(x), \\
	\mu_2(x)  = -\id(x \le \tfrac{1}{4}) - \Big(\tfrac{3}{2}\sin(2 \pi x) + \tfrac{1}{2} \Big) \cdot \id(\tfrac{1}{4} < x \le \tfrac{3}{4}) + 2 \cdot \id(x > \tfrac{3}{4}),
	& \mu_5(x)  = \tfrac{3}{2} - \mu_2(x), \\
	\mu_3(x)  = \id(x > \tfrac{1}{2}),
	& \mu_6(x)  = 1 - \mu_3(x). \\
\end{array}
\]
The functions were selected to be monotonous and non-monotonous, smoothly and abrupt, increasing and decreasing, as displayed in Figure \ref{fig:mu}.
\begin{figure}
	\centering
	\includegraphics[width=\linewidth]{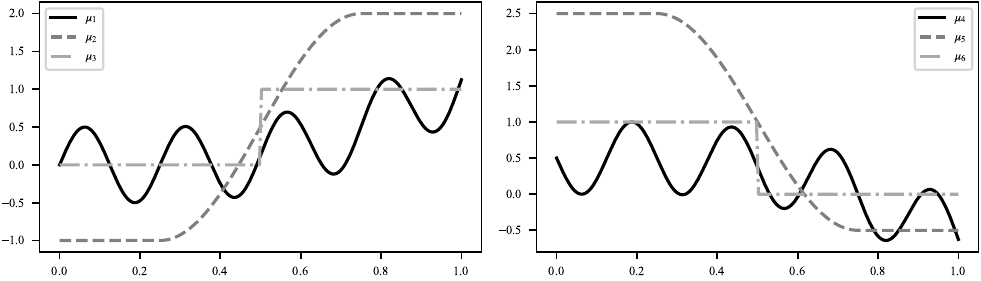}
	\caption{Various mean functions, used to generate time series under the alternative.}
	\label{fig:mu}
\end{figure}
Similarly, for $\sigma^2$, we considered
\[
\begin{array}{ll}
	\sigma_0(x)  = \tfrac{1}{2},
	& \sigma_1(x)  = \tfrac{1}{4} + \tfrac{1}{2} x, \\
	\sigma_2(x)  = \tfrac{1}{2} - \tfrac{1}{4} \cos(2 \pi x),
	& \sigma_3(x)  = \tfrac{1}{4} + \tfrac{1}{2} \id(x > \tfrac{1}{2}). \\
\end{array}
\]
Finally, as error processes, a sequence of i.i.d. random variables, $MA(1)$ and $AR(1)$ processes, as well as a locally stationary process were considered. More specifically, for $(\eta_i)_{i\in\Z}$ with $\eta_i \sim \Nc(0, 1)$ i.i.d., we considered
\[ 
	(\mathrm{iid})~ \eps_i = \eta_i, \quad \quad (\mathrm{ma})~ \eps_i = \tfrac{2}{\sqrt{5}}(\eta_i + \tfrac{1}{2} \eta_{i-1}), \quad \quad (\mathrm{ar})~ \eps_i = \tfrac{\sqrt{3}}{2}(\eta_i + \tfrac{1}{2} \eps_{i-1}),
\]
and
\[
	(\mathrm{ls})~\eps_{i, n} = \sqrt{a(i/n)} \eps_i^{(1)} + \sqrt{1 - a(i/n)} \eps_i^{(2)}, 
\]
where $\eps_i^{(1)}$ is an $AR(1)$ process as before, $\eps_i^{(2)}$ is an $AR(1)$ process with uniform i.i.d. innovations $(\tilde{\eta}_i)_{i\in\Z}$, with $\ex[\tilde{\eta}_i] = 0$ and $\var(\tilde{\eta}_i) = 1$, satisfying $\eps_i^{(2)} = \tfrac{\sqrt{3}}{2}(\eta_i - \tfrac{1}{2} \eps_{i-1}^{(2)})$, and
\[
	a(t) = \tfrac{1}{2}\big[1 - \cos\big(\tfrac{\pi}{2}[-\cos(\pi t) + 1]\big) \big].
\]
Exemplary trajectories of an $AR(1)$ process for $\sigma_2$ and $\sigma_3$ under $H_0$, for the constant mean function $\mu(x) \equiv 0$, are displayed in Figure \ref{fig:sigma}.

\begin{figure}
	\centering
	\includegraphics[width=\linewidth]{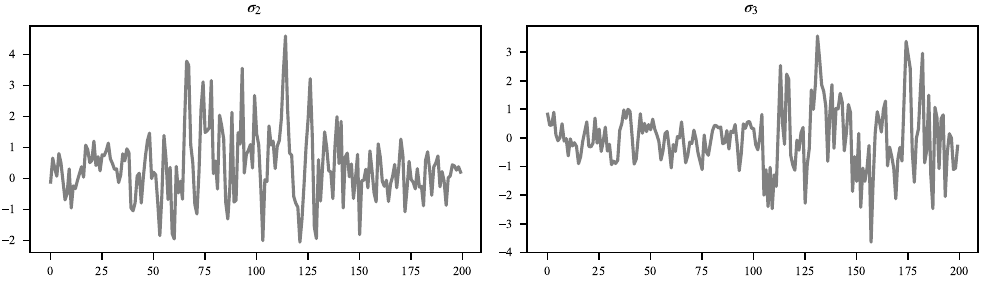}
	\caption{Exemplary trajectories of $AR(1)$ processes with $\sigma= 2\sigma_2$ (left) and $\sigma= 2\sigma_3$ (right), for $\mu(x) \equiv 0$ and $n=200$.}
	\label{fig:sigma}
\end{figure}

For all settings, we generated $1000$ time series and test $H_0$ with level $\alpha = 5\%$. Table \ref{tab:eps0} in the appendix contains  empirical rejection rates under the null hypothesis for different choices of $\eps$ and $\sigma$, while Table \ref{tab:eps1} in the appendix contains those values under the alternative $\mu_5$. Table \ref{tab:mu} displays results for all choices of $\mu$, covering both the null hypothesis and different alternatives.

First consider the null hypothesis $\mu = \mu_0$. It can be seen that \texttt{R1}, \texttt{R2}, \texttt{BT} and \texttt{LRV} exceed the level $\alpha = 0.05$ substantially, which only slightly (if at all) improves for larger values of $n$.  \eqref{eq:decision_rule0} seems to have (approximately) the level $\alpha = 5\%$ for i.i.d. errors, but exceeds the level when the dependence increases. Only the self-normalization based tests, \texttt{SN} and \eqref{eq:decision_rule}, have levels of approximately $\alpha = 0.05$.

Regarding the alternative $\mu_5$, as expected, the results are (partially) reversed. \texttt{SN} generally has the least power across all tests. The tests, that exceed the nominal level $\alpha = 0.05$ under the null hypothesis, reject the null correctly in $100\%$ of the cases. However, more interestingly, \eqref{eq:decision_rule} has empirical rejection rates well above $95\%$ for $n=200$ and $100\%$ for $n \ge 500$. With these empirical rejection rates, \eqref{eq:decision_rule} has substantially more power than \texttt{SN}, where results of the latter vary widely between $0.6\%$ and $100\%$. This effect even holds across different alternatives, as illustrated by the results in Table \ref{tab:mu}.

Table \ref{tab:computation_time} provides average computation times for the different tests. Overall, \texttt{LRV} has the lowest average computation time and seems to scale best. Among the other tests, for short time series with $n=200$, the proposed tests \eqref{eq:decision_rule0} and \eqref{eq:decision_rule} require the least time. As expected, the bootstrap procedure has the highest computation time. Generally though, the computing time for all tests is negligible, with a maximum of $7.3$\,ms assumed by \texttt{BT} for $n=1000$.

In summary, \texttt{R1}, \texttt{R2}, \texttt{BT} and \texttt{LRV} seem unsuitable to detect changes in the considered context of a varying long-run variance, as they exceed the specified level $\alpha$. If the errors cannot be assumed to be independent, \eqref{eq:decision_rule0} might exceed the level too. Out of the tests that have level $\alpha$, \eqref{eq:decision_rule} has by far the highest power under all considered alternatives, and is the preferred test for the detection of changes for locally stationary time series.

\subsubsection{Local Alternatives}

In addition to the simulation study with fixed alternatives, we considered local alternatives as described in Section \ref{sec:local_alternatives}. More specifically, we considered
\begin{align*}
	\mu_{\mathrm{abrupt}}(t) & = a_n \id(t \le \tfrac{1}{2}) \\
	\mu_{\mathrm{smooth}}(t) & = a_n h\Big(\tfrac{t - \tfrac{1}{4}}{0.1 n}\Big),
\end{align*}
where $h(t) = \tfrac{15}{16}(1-t^2)^2$ is the quartic kernel, with $n = 500$, locally stationary errors and $\sigma_3$, as described previously. As before, we generated $1000$ time series for each model and test $H_0$ with level $\alpha = 5\%$. Figure \ref{fig:local_alternatives} displays the empirical rejection rates of the different tests for varying values of $a_n$ in $[-32, 32]$ with logarithmic $x$-axis. Precise results are given in Tables \ref{tab:local_abrupt} and \ref{tab:local_bump} in the appendix.

Generally, it seems more difficult to detect local smooth alternatives compared to abrupt alternatives. As in the case of fixed alternatives, only \texttt{SN} and  \eqref{eq:decision_rule} have nominal levels below $5\%$ under the null hypothesis, for $a_n = 0$. As before, \eqref{eq:decision_rule} has substantially more power than \texttt{SN}.
Interestingly, \texttt{R1}, \texttt{R2} and \texttt{SN}, which are based on a local linear estimation of $\mu$ have vanishing power for large values of $|a_n|$. This can be expected, since a large jump contradicts the underlying assumption of smoothness of $\mu$, required for the local linear estimator.

\begin{figure}
	\centering
	\includegraphics[width=\linewidth]{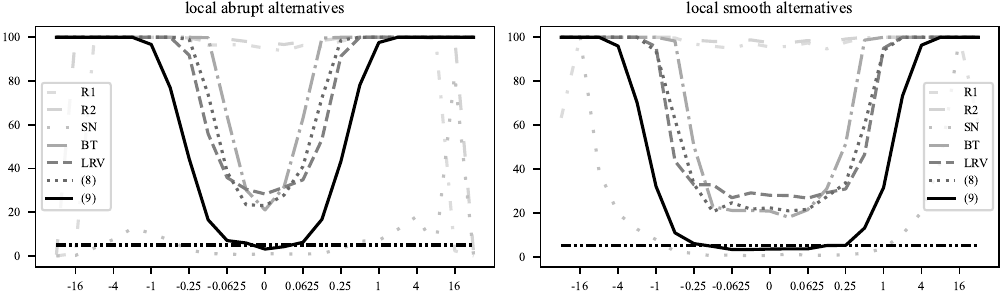}
	\caption{Empiricial rejection rates ($y$-axis) across different values of $a_n$ (logarithmic $x$-axis) for local alternatives. The horizontal line at $5\%$ indicates the nominal level under the null hypothesis.}
	\label{fig:local_alternatives}
\end{figure}

\subsection{Case Study}

\textit{Temperature Curves.} Time series with possibly varying mean, variance and dependence structure occur naturally in meteorology. We consider the mean of daily minimal temperatures (in degrees Celsius) over the month of July for a period of approximately 120 years across eight places in Australia.\footnote{The data is freely available from the Bureau of Meteorology of the Australian Government at \url{https://www.bom.gov.au/climate/data/index.shtml}.} Exemplary, the recorded temperature curves at the weather stations in Gayndah, Robe and Sydney are plotted in Figure \ref{fig:australia}. 

The results for all weather stations, given in terms of $p$-values, are displayed in Table \ref{tab:case_study}. The tests \texttt{BT} and \texttt{LRV} have $p$-values well below $0.05$ across all stations, indicating a change in the temperature. Contrarily, the test \texttt{SN} has $p$-values between $0.4$ and $0.5$, so that the null hypothesis of no change cannot be rejected. More interestingly, the results for \texttt{R1} and \eqref{eq:decision_rule} oppose in certain locations. For example, in Gayndah, \eqref{eq:decision_rule} has a $p$-value of $0.003$, which is highly significant at a level of $5\%$, whereas \texttt{R1} has a $p$-value of $0.832$ in the same location. Conversely, the latter has a significant $p$-value in Cape Otway, whereas the former has a corresponding value of $0.262$. 

The difference between \texttt{R1} and \eqref{eq:decision_rule} might be explained by the different approaches. While \texttt{R1} estimates the mean locally and detects local deviations from the mean, \eqref{eq:decision_rule} calculates a global statistic through cumulative sums. As displayed in Figure \ref{fig:australia}, the temperature in Cape Otway varies only slightly across the entire time horizon, but deviates substantially from typical temperatures at the very beginning. Contrarily, the mean temperature in Gayndah varies minimally throughout time, and not ``relevantly'' in a short interval. Both tests have low $p$-values between $0.1$ and $0.15$ in Melbourne, where the mean temperature increases to a greater extent.

\begin{figure}
	\centering
	\includegraphics[width=\linewidth]{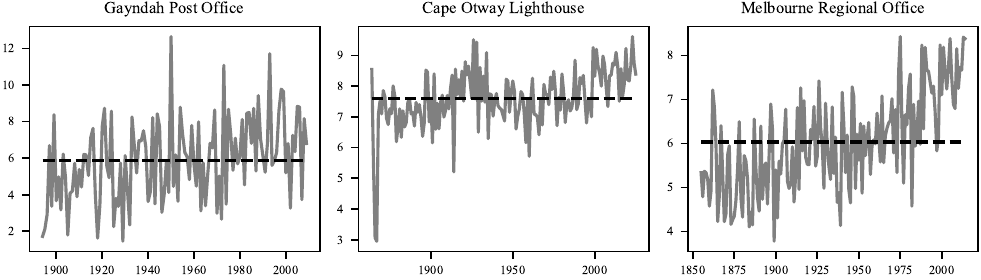}
	\caption{Mean temperatures in Gayndah (left), Cape Otway (center) and Melbourne (right) for the month of July (gray), jointly with overall averages (black).}
	\label{fig:australia}
\end{figure}

\setlength{\tabcolsep}{5pt}
	
\begin{table} 
	\caption{$p$-values of tests across different locations. Significant p-values (below 0.05) are in boldface.}  \label{tab:case_study}
\begin{tabular}{l|rrrrrr}
	\toprule
	& R1 & R2 & SN & BT & LRV & \eqref{eq:decision_rule}  \\
	\midrule
	Boulia Airport 			  & 0.680 & 0.728 & 0.487 & \textbf{0.000} & \textbf{0.000} & 0.347 \\
	Gayndah Post Office 	  & 0.832 & 0.392 & 0.508 & \textbf{0.000} & \textbf{0.000} & \textbf{0.003} \\
	Gunnedah Pool 			  & 0.657 & 0.513 & 0.480 & \textbf{0.002} & \textbf{0.000} & 0.441 \\
	Hobart TAS 				  & 0.367 & 0.102 & 0.484 & \textbf{0.000} & \textbf{0.000} & 0.467 \\
	Melbourne Regional Office & 0.146 & \textbf{0.000} & 0.411 & \textbf{0.000} & \textbf{0.000} & 0.112 \\
	Cape Otway Lighthouse 	  & \textbf{0.024} & \textbf{0.000} & 0.499 & \textbf{0.000} & \textbf{0.000} & 0.262 \\
	Robe 					  & 0.056 & \textbf{0.000} & 0.410 & \textbf{0.000} & \textbf{0.000} & 0.341 \\
	Sydney 					  & 0.155 & \textbf{0.009} & 0.466 & \textbf{0.000} & \textbf{0.000} & 0.358 \\
	\bottomrule
\end{tabular}

\end{table}

\textit{EEG Data.} Another example of possibly non-stationary time series comes from neuroscience. Brain activity is often recorded using electroencephalography (EEG), whereby electrodes are attached to the scalp to measure voltages. The recorded signals may be non-stationary for various reasons, for example because the impedance changes when electrodes move. 
In the following, we consider the ``Consumer-grade EEG-based Eye Tracking'' dataset, which contains approximately 12 hours of EEG recordings from 113 subjects \citep{afonso2025}. The preprocessing steps suggested by the authors were used. The dataset contains different ``tasks'' and only ``level-2-smooth'' recordings were considered for the experiments, since this was the largest category.

Table \ref{tab:eeg} displays the empirical rejection rates and mean $p$-values for the considered tests, based on the 102 EEG recordings without technical problems. The tests based on local linear estimation (\texttt{R1}, \texttt{R2} and \texttt{SN}), as well as \eqref{eq:decision_rule}, have empirical rejection rates below $2\%$ and considerably large mean $p$-values. Tests \texttt{BT} and \texttt{LRV} reject the null hypothesis of a constant mean for 20.6\% and 18.6\% of the EEG recordings. Generally it seems that the majority of recordings have a constant mean and only a small proportion exhibits a drift.

\begin{table} 
	\caption{Empirical rejection rates and average $p$-values of tests across EEG recordings.}  \label{tab:eeg}
	\begin{tabular}{l|rrrrrr}
		\toprule
		& R1 & R2 & SN & BT & LRV & \eqref{eq:decision_rule} \\
		\midrule
		Empirical Rejection Rate & 0.000 & 0.000 & 0.000 & 0.206 & 0.186 & 0.020 \\
		Mean $p$-value & 0.837 & 0.763 & 0.494 & 0.382 & 0.497 & 0.452 \\
		\bottomrule
	\end{tabular}
	
\end{table}

\section{Conclusion} \label{sec:conclusion}

A self-normalized test statistic, based on the CUSUM process, has been proposed for the detection of changes in the mean. In contrast to prior work, assumptions on the mean function $\mu$ have been relaxed. In a simulation study, the proposed test and the test by \cite{heinrichs2021} were found to be the only ones with empirical rejection rates close to the level $\alpha$ under the null hypothesis. Compared to the latter, the proposed test was found to be substantially more powerful. 

Similarly to the detection of changes in $\mu(i/n) = \ex[X_{i, n}]$, one may use the same approach for the detection of changes in $\ex[X_{i,n}^2]$. More generally, one may test the constancy of $\ex[f(X_{i,n})]$ for arbitrary real-valued functions $f$, whenever the test's assumptions are satisfied for $\{f(X_{i, n})_{1 \le i \le n}\}_{n\in\N}$.  

If we test for the constancy of $\mu(i/n)$ and $\ex[X_{i,n}^2]$, we can combine both quantities to $\var(X_{i, n})$. Similarly, we might test if the observations are uncorrelated, by testing the null hypothesis $\cov(X_{i, n}, X_{i+h, n}) = 0$, for $i=1, \dots, n-h$. Note that in this case, as we conduct multiple tests, we have to control the joint level $\alpha$ by reducing the level of each individual test. In future work, it might be  worthwhile to extend the proposed methodology to multivariate time series. This would allow a simultaneous test for multiple autocovariances, or a Portmanteau-type test \citep[see, e.\,g.,][]{bucher2023}. Another interesting extension would be a generalization to functional data, in which case the estimation of the long-run variance becomes even more difficult. 

Finally, the idea behind the ``double-indexed'' process $S_n(t, s)$ might be transferred to extreme value theory, where it could be a starting point for generalizing the self-normalization by \cite{bucher2024} to a broader class of non-stationary time series and may prove useful in other inference problems for locally stationary processes.

 \section{Proofs} \label{sec:proof}

\subsection{Auxiliary Results} \label{sec:auxiliary_results}

For a probability space $(\Omega, \Ac, \pr)$, we will denote the norm of $L^2(\Omega, \Ac, \pr)$ by $\| X \|_\Omega = \ex[X^2]^{1/2}$, for a real-valued random variable $X$, in case of existence. Before proving the lemmas, we collect some useful properties of the physical dependence measure.

\begin{proposition} \label{prop:m-dependence}
	Let Assumption \ref{assump:error} be satisfied, and $\tilde{\eps}_{i, n} = \ex[\eps_{i, n}| \eta_i, \dots, \eta_{i-m}]$. Then, 
	\[\sup_{1\le i\le n, n\in \N} \| \eps_{i, n} - \tilde{\eps}_{i, n}\|_\Omega \le \Theta_m.\]
\end{proposition}

\begin{proof}
	First note that $\tilde{\eps}_{i, n}$ is the projection of $\eps_{i, n}$ onto the subspace of $\sigma(\eta_i, \dots, \eta_{i-m})$-measurable random variables in $L^2(\Omega, \Ac, \pr)$. By the Hilbert projection theorem, it minimizes the $L^2$ distance to $\eps_{i, n}$, so that
	\begin{equation}\label{eq:min_l2_dist}
		\| \eps_{i, n} - \tilde{\eps}_{i, n} \|_\Omega \le \| \eps_{i, n} - Z\|_\Omega,
	\end{equation}
	for any $\sigma(\eta_i, \dots, \eta_{i-m})$-measurable random variable $Z$. Further recall that $\eta^* = (\eta_i^*)_{i\in\Z}$ is an independent copy of $\eta = (\eta_i)_{i\in\Z}$, and let $\eps_{i, n}^* = H(i/n, \Fc_i^m)$, where 
	\[ \Fc_i^m = (\dots, \eta_{i-m-2}^*, \eta_{i-m-1}^*, \eta_{i-m}, \dots, \eta_i). \]
	Clearly, $\eps_{i, n}^*$ is independent of $(\eta_k)_{k \le i-m-1}$, so that
	\[ \ex[\eps_{i, n}^* |\eta_i, \dots, \eta_{i-m}] = \ex[\eps_{i, n}^* |\Fc_i]. \]
	Moreover, $\eps_{i, n}$ is measurable with respect to $\sigma(\Fc_i)$, so that $\eps_{i, n} =  \ex[\eps_{i, n}|\Fc_i]$. With $Z = \ex[\eps_{i, n}^*| \eta_i, \dots, \eta_{i-m}]$, it follows from \eqref{eq:min_l2_dist} that

	\begin{equation*}
		\| \eps_{i, n} - \tilde{\eps}_{i, n}\|_\Omega
		 \le \big\| \eps_{i, n} - \ex[\eps_{i, n}^*| \eta_i, \dots, \eta_{i-m}] \big\|_\Omega 
		 = \big\| \ex[\eps_{i, n}-\eps_{i, n}^*|\Fc_i]\big\|_\Omega.
	\end{equation*}
	Since the conditional expectation is a contraction, the right-hand side can be bounded from above by $\| \eps_{i, n}-\eps_{i, n}^*\|_\Omega$. From the expansion
	\[ \eps_{i, n}-\eps_{i, n}^* = \sum_{k=m}^\infty H(i/n, \Fc_i^{k+1}) - H(i/n, \Fc_i^k), \]
	for $1\le i\le n$, and the triangle inequality, we obtain
	\begin{equation} \label{eq:bound_mn_dep}
		\| \eps_{i, n}-\eps_{i, n}^*\|_\Omega \le \sum_{k=m}^\infty \big\| H(i/n, \Fc_i^{k+1}) - H(i/n, \Fc_i^k)\big\|_\Omega \le \Theta_m.
	\end{equation}
	The last bound holds uniformly for $1\le i\le n$ and $n\in\N$, since 
	\begin{equation*}
		\sup_{1\le i\le n, n\in \N} \big\| H(i/n, \Fc_i^{k+1}) - H(i/n, \Fc_i^k)\big\|_\Omega \le \delta(H, k+1).
	\end{equation*}

\end{proof}

\begin{proposition} \label{prop:lrv_approximation}
	Let Assumption \ref{assump:error} be satisfied. Further, let $a_n$ and $b_n$ be sequences such that $a_n, b_n \to \infty$. Then, 
	\[ \sum_{i=1}^{a_n} \sum_{j=1}^{b_n} \cov\big(H(t, \Fc_{i-j}), H(t, \Fc_0)\big) = (a_n \wedge b_n)  \sigma^2(t) + o(a_n \wedge b_n). \]
\end{proposition}

\begin{proof}
	In the following, denote the covariance $\cov\big(H(t, \Fc_{h}), H(t, \Fc_0)\big)$ by $\gamma_h$. Note that $\gamma_h$ is symmetric in $h$, so that $\gamma_h = \gamma_{-h}$, for $h\in\Z$. By an index shift and changing the order of summation, we have
	\[ 
		\sum_{i=1}^{a_n} \sum_{j=1}^{b_n} \gamma_{i-j} 
		= \sum_{j=1}^{b_n} \sum_{i=1-j}^{a_n-j} \gamma_{i}
		= \sum_{i=1-b_n}^{a_n-1} \sum_{j=(1 - i)\vee 1}^{(a_n-i)\wedge b_n} \gamma_{i}  
		= \sum_{i=1-b_n}^{a_n-1} \big( [(a_n-i)\wedge b_n] + [i \wedge 0] \big) \gamma_{i}.  
	\]
	Splitting the right-hand side into sums with positive and negative summation indices, and using the symmetry of $\gamma_{i}$, we further obtain
	\begin{align} \label{eq:lrv_approximation}
		& \sum_{i=1}^{a_n-1} [(a_n-i)\wedge b_n] \gamma_{i}
		+ \sum_{i=1}^{b_n-1} [(b_n-i)\wedge a_n] \gamma_{i}
		+ (a_n \wedge b_n) \gamma_{0} \\
		& = \sum_{i=1}^{(a_n \wedge b_n)-1} (a_n \wedge b_n -i) \gamma_{i}
		+ \sum_{i=1}^{(a_n \vee b_n)-1} [(a_n \vee b_n-i)\wedge (a_n \wedge b_n)] \gamma_{i}
		+ (a_n \wedge b_n) \gamma_{0}. \notag
	\end{align}
	The term $(a_n \vee b_n-i)\wedge (a_n \wedge b_n)$ can be written as $(a_n \wedge b_n) + [0 \wedge (|a_n - b_n|-i)]$, where the second summand is non-zero whenever $i > |a_n - b_n|$. Hence, again by symmetry of $\gamma_i$, we can simplify the right-hand side of \eqref{eq:lrv_approximation} as
	\[ (a_n \wedge b_n) \sum_{i=-(a_n \vee b_n) + 1}^{(a_n \wedge b_n)-1} \gamma_{i}
	- \sum_{i=1}^{(a_n \wedge b_n)-1} i \gamma_{i}
	+ \sum_{i=|a_n-b_n|+1}^{(a_n \vee b_n)-1} (|a_n - b_n| - i) \gamma_{i}. \]
	By assumption, the series $\sum_{h=-\infty}^\infty \gamma_h = \sigma^2(t)$ converges. By standard arguments, it follows that the first sum equals $(a_n \wedge b_n)  \sigma^2(t) + o(a_n \wedge b_n)$, whereas the other two sums are of order $o(a_n \wedge b_n)$.
\end{proof}

\subsection{Proof of Lemma \ref{lem:fdd_conv}} \label{sec:proof_fdd_conv}

Before giving the rigorous proof, we briefly summarize its line of reasoning. First, we use the Cramér-Wold device to reduce the statement to a univariate convergence. Next, we show that the last $n - b_n \ell_n$ random variables are asymptotically negligible. Then, we replace the error process $\{(\eps_{i,n})_{i=1, \dots, n}\}_{n\in\N}$ by $m_n$-dependent random variables $\{(\tilde{\eps}_{i,n})_{i=1, \dots, n}\}_{n\in\N}$, using Proposition \ref{prop:m-dependence}. We rewrite the process $G_n$ in terms of a double-sum, given by the blocks from the definition of the permutation $\pi$. Using the usual big-blocks-small-blocks technique, we show that the small blocks are asymptotically negligible and the big blocks are asymptotically independent. Classic arguments for Riemann sums and Proposition \ref{prop:lrv_approximation}, yield the required covariance structure. Finally, moment bounds for the errors allow us to apply Lyapunov's central limit theorem.

\textbf{Proof:} By the Cramér-Wold device, \eqref{eq:fdd_conv} is equivalent to
\begin{equation*}
	\sum_{i=1}^{d} a_i G_n(t_i, s_i) \convw \sum_{i=1}^{d} a_i G(t_i, s_i),
\end{equation*}
for all $a_1, \dots, a_d\in\R$. 
The left-hand side of the previous display may be written as 
\[
	\sum_{i=1}^{d} a_i G_n(t_i, s_i) = \sum_{i=1}^{d} a_i \frac{1}{\sqrt{n}}\sum_{j=1}^{\ell_n b_n} \eps_{\pi_j, n} \id(j \le \lfloor t_i n \rfloor, \pi_j \le \lfloor s_in \rfloor) + R_n, 
\]
with remainder 
\[
	R_n  = \sum_{i=1}^{d} a_i \frac{1}{\sqrt{n}}\sum_{j=\ell_n b_n+1}^{n} \eps_{\pi_j, n} \id(j \le \lfloor t_i n \rfloor, \pi_j \le \lfloor s_i n \rfloor).
\]
The remainder is asymptotically negligible, since 
\begin{equation} \label{eq:remainder}
	\ex[R_n^2] \le \sum_{i_1, i_2=1}^{d} a_{i_1} a_{i_2} \frac{n-\ell_n b_n}{n} \sum_{j=\ell_n b_n+1}^n \ex[\eps_j^2] = \Oc(b_n^2 n^{-1})
\end{equation}
by Jensen's inequality and part 2 of Assumption \ref{assump:error}. 

For $m_n$ as in Assumption \ref{assump:sequences} and $i = 1, \dots, n$, define 
\begin{equation} \label{eq:m_dependent_rvs}
	\tilde{\eps}_{i, n} = \ex[\eps_{i, n} | \eta_i ,\dots, \eta_{i-m_n} ].
\end{equation}
By Proposition \ref{prop:m-dependence}, $\sup_{1\le i\le n, n\in \N} \|\eps_{i, n} - \tilde{\eps}_{i, n}\|_\Omega \le \Theta_{m_n}$, so that

\begin{align} \label{eq:m_dependence_approx}
	&\bigg\| \sum_{i=1}^{d} a_i \frac{1}{\sqrt{n}}\sum_{j=1}^{\ell_n b_n} \big( \eps_{\pi_j, n} - \tilde{\eps}_{\pi_j, n}\big) \id(j \le \lfloor t_i n \rfloor, \pi_j \le \lfloor s_in \rfloor) \bigg\|_\Omega \\
	& \le \sum_{i=1}^{d} a_i \frac{1}{\sqrt{n}}\sum_{j=1}^{\ell_n b_n} \| \eps_{\pi_j, n} - \tilde{\eps}_{\pi_j, n}\|_\Omega = \Oc(\sqrt{n}\Theta_{m_n}). \notag
\end{align}
Hence, we can rewrite 
\[
	\sum_{i=1}^{d} a_i G_n(t_i, s_i) = \sum_{i=1}^{d} a_i \frac{1}{\sqrt{n}}\sum_{j=1}^{\ell_n b_n} \tilde{\eps}_{\pi_j, n} \id(j \le \lfloor t_i n \rfloor, \pi_j \le \lfloor s_in \rfloor) + \Oc_\pr(b_n n^{-1/2} + \sqrt{n}\Theta_{m_n}). 
\]
By definition of the permutation $\pi$, we have
\begin{align*}
	& \sum_{i=1}^{d} a_i \frac{1}{\sqrt{n}}\sum_{j=1}^{\ell_n b_n} \tilde{\eps}_{\pi_j, n} \id(j \le \lfloor t_i n \rfloor,\; \pi_j \le \lfloor s_i n \rfloor) \\
	& = \sum_{i=1}^{d} a_i \frac{1}{\sqrt{n}}\sum_{k=1}^{b_n} \sum_{j=1}^{\ell_n} \tilde{\eps}_{k+(j-1)b_n, n} \id\big( (k-1)\ell_n + j \le \lfloor t_i n \rfloor,\; k+(j-1)b_n \le \lfloor s_i n \rfloor\big).
\end{align*}
Changing the order of summation and defining
\[ Y_{k, j} = \sum_{i=1}^{d} a_i \tilde{\eps}_{k+(j-1)b_n, n} \id\big( (k-1)\ell_n + j \le \lfloor t_i n \rfloor,\; k+(j-1)b_n \le \lfloor s_i n \rfloor\big), \]
for $k=1, \dots, b_n$ and $j = 1, \dots, \ell_n$, we can rewrite the right-hand side of the previous display as $\frac{1}{\sqrt{n}}\sum_{k=1}^{b_n} \sum_{j=1}^{\ell_n} Y_{k, j}$. Note that two random variables $Y_{k_1, j_1}$ and $Y_{k_2, j_2}$ are independent whenever $| k_1 - k_2 +(j_1-j_2)b_n| > m_n$. 

In the following, we split the overall sum into sums of big and small blocks, so that the small blocks are asymptotically negligible and the big blocks are independent. We conclude the lemma's proof by proving the Lyapunov condition and deriving a central limit theorem for the big blocks.

More specifically, define the big blocks
\[ U_j = \{ k \in \N: (j-1) b_n + 1 \le k \le jb_n - m_n \}  \]
and similarly small blocks
\[ V_j = \{ k \in \N: jb_n - m_n + 1 \le k \le jb_n \}  \]
for $j = 1, \dots, \ell_n$. Note that the distance between observations in different big blocks is larger than $m_n$, and the same is true for observations in different small blocks. Hence, observations in different blocks are independent. Further note that $\ex[Y_{k,j}]=0$, for any $k=1, \dots, b_n$ and $j = 1, \dots, \ell_n$. Hence, for the big blocks, it holds
\begin{equation*}
	\Ub_n:= \ex\bigg[\bigg( \frac{1}{\sqrt{n}} \sum_{j=1}^{\ell_n} \sum_{k: (j-1)b_n + k \in U_j} Y_{k, j}\bigg)^2\bigg]
	= \frac{1}{n} \sum_{j=1}^{\ell_n} \sum_{k_1=1}^{b_n-m_n} \sum_{k_2=1}^{b_n-m_n} \ex[Y_{k_1, j} Y_{k_2, j}].
\end{equation*}
Denote by $\gamma_k(t)$ the covariance $\cov\big( H(t, \Fc_k), H(t, \Fc_0) \big)$. By assumption, $H$ is Lipschitz continuous with respect to the $L^2$-norm, so that 
\[  
	\sup_{\substack{j=1, \dots, \ell_n\\ k=1, \dots, b_n}} \Big\| \eps_{(j-1)b_n+k, n} -  H\big( \tfrac{j-1}{\ell_n}, \Fc_{(j-1)b_n+k} \big) \Big\|_\Omega  
	\le C \frac{b_n}{n}. 
\]
Hence, by Proposition \ref{prop:m-dependence} and boundedness of the moments, it holds
\begin{align} 
	& \sup_{\substack{j=1, \dots, \ell_n\\ k_1, k_2=1, \dots, b_n}} \big| \ex[\tilde{\eps}_{(j-1)b_n+k_1, n}\tilde{\eps}_{(j-1)b_n+k_2, n}] - \gamma_{k_1-k_2}\big(\tfrac{j}{\ell_n}\big) \big|  \notag \\
	& = \sup_{\substack{j=1, \dots, \ell_n\\ k_1, k_2=1, \dots, b_n}} \big| \ex[\tilde{\eps}_{(j-1)b_n+k_1, n}\tilde{\eps}_{(j-1)b_n+k_2, n}] - \ex[\eps_{(j-1)b_n+k_1, n}\eps_{(j-1)b_n+k_2, n}] \big| + \Oc\big(\tfrac{b_n}{n}\big) \label{eq:m_dependent_cov} \\
	& \le \big( \sup_{i=1, \dots, n} \| \tilde{\eps}_{i, n} \|_\Omega + \sup_{i=1, \dots, n} \| \eps_{i, n} \|_\Omega \big) \sup_{i=1, \dots, n} \| \eps_{i, n}- \tilde{\eps}_{i, n}\|_\Omega  + \Oc\big(\tfrac{b_n}{n}\big)  =  \Oc\big(\tfrac{b_n}{n} + \Theta_{m_n}\big). \notag
\end{align}
By expanding $Y_{k, j}$ and plugging $\gamma_{k_1-k_2}\big(\tfrac{j}{\ell_n}\big)$ in, we can rewrite 
\begin{equation} \label{eq:big_blocks2}
	\Ub_n = \sum_{i_1, i_2=1}^d a_{i_1} a_{i_2} \frac{1}{n} \sum_{j=1}^{\ell_n} \sum_{k_1=1}^{b_n-m_n} \sum_{k_2=1}^{b_n-m_n} \gamma_{k_1-k_2}\big(\tfrac{j}{\ell_n}\big) A_{k_1, j}(t_1, s_1) A_{k_2, j}(t_2, s_2) + \Oc\big(\tfrac{b_n^2}{n} + b_n\Theta_{m_n}\big),
\end{equation}
where $A_{k, j}(t, s) = \id\big( (k-1)\ell_n + j \le \lfloor t n \rfloor,\; k+(j-1)b_n \le \lfloor s n \rfloor\big)$, for $t, s \in [0, 1], k = 1, \dots, b_n, j = 1, \dots, \ell_n$. Rearranging terms in the indicators yields
\begin{align}\label{eq:indicator}
	A_{k_1, j}(t_1, s_1) A_{k_2, j}(t_2, s_2) = & \id\Big(k_1 \le \frac{\lfloor t_{i_1} n\rfloor -j}{\ell_n} + 1,\; k_2 \le \frac{\lfloor t_{i_2} n\rfloor -j}{\ell_n} + 1\Big)\\
	& \times \id \Big(j \le \frac{(\lfloor s_{i_1} n\rfloor - k_1) \wedge (\lfloor s_{i_2} n\rfloor - k_2)}{b_n} + 1 \Big).\notag
\end{align}
Since 
\[ \frac{\lfloor s_{i_1} n\rfloor  \wedge \lfloor s_{i_2} n\rfloor}{b_n} - \frac{(\lfloor s_{i_1} n\rfloor - k_1) \wedge (\lfloor s_{i_2} n\rfloor - k_2)}{b_n} \le \frac{k_1 \vee k_2}{b_n} \le \frac{b_n -m_n}{b_n} < 1, \]
it exists at most one $j \in \{1, \dots, \ell_n\}$ such that the second indicator on the right-hand side of \eqref{eq:indicator} does not equal $\id \Big(j \le \frac{\lfloor (s_{i_1} \wedge s_{i_2}) n\rfloor}{b_n} + 1 \Big)$. In particular, when replacing the indicator in \eqref{eq:big_blocks2}, the error is of order $\Oc(b_n^2/n)$, so that we can rewrite 
\begin{align} \notag
	\Ub_n = \sum_{i_1, i_2=1}^d a_{i_1} a_{i_2} \frac{1}{n} \sum_{j=1}^{\ell_n} \sum_{k_1=1}^{b_n-m_n} \sum_{k_2=1}^{b_n-m_n} & \gamma_{k_1-k_2}\big(\tfrac{j}{\ell_n}\big) \id\Big(k_1 \le \frac{\lfloor t_{i_1} n\rfloor -j}{\ell_n} + 1,\; k_2 \le \frac{\lfloor t_{i_2} n\rfloor -j}{\ell_n} + 1\Big) \\
	& \times \id \Big(j \le \frac{\lfloor (s_{i_1} \wedge s_{i_2}) n\rfloor}{b_n} + 1 \Big)  + \Oc\big(\tfrac{b_n^2}{n} + b_n\Theta_{m_n}\big).  \label{eq:big_blocks3}
\end{align}
By Proposition \ref{prop:lrv_approximation}, we have
\begin{align*}
	& \sum_{k_1=1}^{b_n-m_n} \sum_{k_2=1}^{b_n-m_n} \gamma_{k_1-k_2}\big(\tfrac{j}{\ell_n}\big)  \id\Big(k_1 \le \frac{\lfloor t_{i_1} n\rfloor -j}{\ell_n} + 1,\; k_2 \le \frac{\lfloor t_{i_2} n\rfloor -j}{\ell_n} + 1\Big) \\
	& =  \frac{\lfloor (t_{i_1} \wedge t_{i_2}) n\rfloor -j}{\ell_n}  \sigma^2\big(\tfrac{j}{\ell_n}\big) + o(b_n),
\end{align*}
so that \eqref{eq:big_blocks3} yields 
\begin{align*} 
	\Ub_n = & \sum_{i_1, i_2=1}^d a_{i_1} a_{i_2} \frac{1}{n} \sum_{j=1}^{\ell_n} \frac{\lfloor (t_{i_1} \wedge t_{i_2}) n\rfloor -j}{\ell_n}  \sigma^2\big(\tfrac{j}{\ell_n}\big)  \id \Big(j \le \frac{\lfloor (s_{i_1} \wedge s_{i_2}) n\rfloor}{b_n} + 1 \Big) \\
	& + \Oc\big(\tfrac{b_n^2}{n} + b_n\Theta_{m_n}\big) + o(1). 
\end{align*}
Using a standard argument based on Riemann sums, it follows that 
\begin{align*}
	& \sum_{j=1}^{\ell_n} \frac{\lfloor (t_{i_1} \wedge t_{i_2}) n\rfloor -j}{\ell_n}  \sigma^2\big(\tfrac{j}{\ell_n}\big)  \id \Big(j \le \frac{\lfloor (s_{i_1} \wedge s_{i_2}) n\rfloor}{b_n} + 1 \Big) \\
	& = (t_{i_1} \wedge t_{i_2}) n \int_0^{s_{i_1} \wedge s_{i_2}} \sigma^2(x) \diff x + \Oc(\ell_n)
\end{align*}
since $\sigma^2$ is Lipschitz continuous by assumption. Hence,
\begin{align*}
	\Ub_n = \sum_{i_1, i_2=1}^d a_{i_1} a_{i_2} (t_{i_1} \wedge t_{i_2}) \int_0^{s_{i_1} \wedge s_{i_2}} \sigma^2(x) \diff x + \Oc\big(\tfrac{b_n^2}{n} +\tfrac{\ell_n}{n}+ b_n\Theta_{m_n}\big) + o(1),
\end{align*}
which converges to $\var(\sum_{i=1}^d a_i G(t_i, s_i))$. Analogously, it follows for the small blocks that 
\begin{equation*}
	\ex\bigg[\bigg( \frac{1}{\sqrt{n}} \sum_{j=1}^{\ell_n} \sum_{k: (j-1)b_n + k \in V_j} Y_{k, j}\bigg)^2\bigg]
	= \Oc\Big( \frac{m_n}{b_n} + \frac{m_n^2}{b_n} \Theta_{m_n} \Big), 
\end{equation*}
which vanishes as $n \to\infty$. Hence, the small blocks are negligible and the asymptotic behavior of $\sum_{i=1}^d a_i G_n(t_i, s_i)$ is determined by the big blocks. Finally, since two random variables $Y_{k_1, j_1}$ and $Y_{k_2, j_2}$ are independent whenever $| k_1 - k_2 +(j_1-j_2)b_n| > m_n$,
\[
	\sum_{k_1: (j-1)b_n + k_1 \in U_j} \dots \sum_{k_4: (j-1)b_n + k_4 \in U_j} \ex \bigg[ \prod_{i=1}^4 Y_{k_i, j} \bigg]
\]
has at most $b_n^2 m_n^2$ non-zero summands. Hence, by part 2 of Assumption \ref{assump:error},
\[ 
	\sum_{j=1}^{\ell_n} \ex \bigg[ \bigg( \frac{1}{\sqrt{n}} \sum_{k: (j-1)b_n + k \in U_j} Y_{k, j}\bigg)^4 \bigg]  \le C \frac{b_n^2 m_n^2 \ell_n}{n^2} = \Oc\Big(\frac{b_nm_n^2}{n}\Big)
\]
for some constant $C > 0$. By Lyapunov's central limit theorem, it follows that 
\[ \sum_{i=1}^{d} a_i G_n(t_i, s_i) \convw \Nc\bigg( 0, \var\bigg( \sum_{i=1}^d a_i G(t_i, s_i) \bigg)  \bigg) \stackrel{\mathcal{D}}{=} \sum_{i=1}^d a_i G(t_i, s_i,) \]
so that the lemma's statement follows from the Cramér-Wold device.

\subsection{Proof of Lemma \ref{lem:equicont}} \label{sec:proof_equicont}

As before, we briefly summarize the main arguments of the lemma's proof. By the triangle inequality, stochastic equicontinuity of $G_n$ in both arguments is equivalent to the property in each argument separately, when taking the supremum over the other argument. Further, $G_n$ may be replaced by $\tilde{G}_n$, based on the $m_n$-dependent random variables $\{(\tilde{\eps}_{i,n})_{i=1, \dots, n}\}_{n\in\N}$. By careful inspection of the indices and using moment bounds on the errors, the $L^4$-norm of $\tilde{G}_n(t, s_1) - \tilde{G}_n(t, s_2)$ is bounded from above by $C |s_1 - s_2|^{3/2}$, for all $s_1, s_2 \in [0, 1]$ with $|s_1 - s_2| > (2^{8/3}\ell_n)^{-1}$. Using Lemma A.1 of \cite{kley2016}, $\lim_{\rho \searrow 0} \lim_{n\to\infty}\ex\big[\sup_{t \in [0, 1], |s_1 - s_2| \le \rho} \big(\tilde{G}_n(t, s_1) - \tilde{G}_n(t, s_2)\big)^4\big]^{1/4}$ can be ultimately bounded from above by $C \eta^{4/3}$, for any $\eta > 0$, so that stochastic equicontinuity in $s$ follows by Markov's inequality. Similarly,  $\ex\big[\big(\tilde{G}_n(t_1, s) - \tilde{G}_n(t_2, s) \big)^4\big]^{1/4} \le C  |t_1 - t_2|^{1/2}$, for all $t_1, t_2 \in [0, 1]$ with $|t_1 - t_2| > (4b_n)^{-1}$. Again, using Lemma A.1 of \cite{kley2016}, stochastic equicontinuity in $t$ is derived. Though, a crucial difference in the arguments is, that in the latter case we make use of martingale properties rather than a careful manipulation of the indicators and moment bounds only.

\textbf{Proof:} First, by the triangle inequality,
\[ 
|G_n(t_1, s_1) - G_n(t_2, s_2)| \le |G_n(t_1, s_1) - G_n(t_1, s_2)| + |G_n(t_1, s_2) - G_n(t_2, s_2)|, 
\]
so that the lemma follows from
\begin{align}
	& \lim_{\rho \searrow 0} \lim_{n\to\infty} \pr \Big( \sup_{t \in [0, 1], |s_1 - s_2| \le \rho} |G_n(t, s_1) - G_n(t, s_2)| > \eps \Big) = 0, \label{eq:tight1} \\
	& \lim_{\rho \searrow 0} \lim_{n\to\infty} \pr \Big( \sup_{s \in [0, 1], |t_1 - t_2| \le \rho} |G_n(t_1, s) - G_n(t_2, s)| > \eps \Big) = 0. \label{eq:tight2}
\end{align}
The two convergences are proven essentially by using similar arguments. 

Recall the $m_n$-dependent random variables $\tilde{\eps}_{i, n} = \ex[\eps_{i, n} | \eta_i, \dots, \eta_{i-m_n}]$ from \eqref{eq:m_dependent_rvs}, for a sequence $(m_n)_{n\in\N}$ as in Assumption \ref{assump:sequences}. Similarly to \eqref{eq:remainder} and \eqref{eq:m_dependence_approx}, it holds
\[
\ex\Big[ \sup_{s, t\in[0, 1]}\big( G_n(t, s) - \tilde{G}_n(t, s)  \big)^2 \Big]^{1/2} = \Oc(\sqrt{n} \Theta_{m_n} + b_n n^{-1/2}), 
\]
where $\tilde{G}_n(t, s)$ is defined as
\[
\tilde{G}_n(t, s) = \frac{1}{\sqrt{n}} \sum_{i=1}^{\ell_n b_n} \tilde{\eps}_{\pi_i, n} \id( i \le \lfloor tn \rfloor, \pi_i \le \lfloor sn \rfloor).
\]
Hence, stochastic equicontinuity of $G_n(t, s)$ follows from \eqref{eq:tight1} and \eqref{eq:tight2} for the process $\tilde{G}_n(t, s)$. 
For \eqref{eq:tight1}, consider
\begin{align*}
	\tilde{G}_n(t, s_1) - \tilde{G}_n(t, s_2) 
	= \frac{1}{\sqrt{n}} \sum_{j=1}^{\ell_n} \sum_{k=1}^{b_n} & \tilde{\eps}_{k+(j-1)b_n, n} \id(k \le \tfrac{\lfloor t n \rfloor - j}{\ell_n} + 1) \\
	& \times \Big[ \id(j \le \tfrac{\lfloor s_1 n \rfloor - k}{b_n} + 1) - \id(j \le \tfrac{\lfloor s_2n \rfloor - k}{b_n} + 1) \Big].
\end{align*}
The indicator difference can be expanded to
\begin{align*}
	\operatorname{sign}(s_1 - s_2) \cdot \id\Big(\tfrac{\lfloor (s_1 \wedge s_2) n \rfloor - k}{b_n} + 1 < j \le \tfrac{\lfloor (s_1 \vee s_2) n \rfloor - k}{b_n} + 1\Big) ,
\end{align*}
where $\operatorname{sign}(x) = \id(x > 0) - \id(x < 0)$ specifies the sign of a value $x\in\R$. In the following, we calculate the fourth moment of $|\tilde{G}_n(t, s_1) - \tilde{G}_n(t, s_2)|$. 
Let $A_{k, j} =  \id\big( k \le \tfrac{\lfloor t n \rfloor - j}{\ell_n} + 1 \big)$ and $B_{k, j} = \id\big( \tfrac{\lfloor (s_1 \wedge s_2) n \rfloor - k}{b_n} + 1 < j \le \tfrac{\lfloor (s_1 \vee s_2) n \rfloor - k}{b_n} + 1 \big)$. Then,
\begin{equation*}
	\Mb_n := \ex\big[\big(\tilde{G}_n(t, s_1) - \tilde{G}_n(t, s_2) \big)^4\big]
	= \ex\bigg[\bigg( \frac{1}{\sqrt{n}} \sum_{j=1}^{\ell_n} \sum_{k=1}^{b_n} \tilde{\eps}_{k+(j-1)b_n, n} A_{k, j} B_{k, j} \bigg)^4\bigg].
\end{equation*}
Since the random variables $\tilde{\eps}_{i_1}, \tilde{\eps}_{i_2}$ are centered and independent for $|i_1 - i_2| > m_n$, most moments are zero, when expanding the parenthesis in the expectation. More specifically, $\ex[\prod_{i_\nu=1}^4 \tilde{\eps}_{i_\nu, n}]$ does not vanish, only if 
\begin{itemize}
	\item all random variables are dependent, essentially having the same index $j$, or
	\item there are 2 pairs of 2 dependent random variables, essentially having the 2 (pairwise different) indices $j_1, j_2$.
\end{itemize}
In particular, we can rewrite 
\begin{align*}\
	\Mb_{n} 
	=  \Zb_{n, 1} & + \frac{3}{n^2} \sum_{\substack{j_1,j_2=1\\j_1\neq j_2}}^{\ell_n} \sum_{k_1, k_3 = 1}^{b_n} \sum_{k_2=k_1 - m_n}^{k_1 + m_n}  \sum_{k_4=k_3 - m_n}^{k_3 + m_n}  \ex[\tilde{\eps}_{k_1+(j_1-1)b_n, n} \tilde{\eps}_{k_2+(j_1-1)b_n, n} ]    \notag \\
	& \times \ex[\tilde{\eps}_{k_3+(j_2-1)b_n, n} \tilde{\eps}_{k_4+(j_2-1)b_n, n} ] 
	 \times A_{k_1, j_1} A_{k_2, j_1} A_{k_3, j_2}A_{k_4, j_2} B_{k_1, j_1}B_{k_2, j_1}B_{k_3, j_2}B_{k_2, j_2},
\end{align*}
where
\begin{equation*} 
	\Zb_{n, 1} = \frac{1}{n^2} \sum_{j=1}^{\ell_n} \sum_{k_1=1}^{b_n} \sum_{k_2=1-m_n}^{b_n+m_n} \sum_{k_3=1-2m_n}^{b_n+2m_n} \sum_{k_4=1-3m_n}^{b_n+3m_n} \ex\bigg[\prod_{i=1}^4 \tilde{\eps}_{k_i+(j-1)b_n, n} \bigg]  \prod_{i=1}^4 A_{k_i, j} B_{k_i, j}. 
\end{equation*}
Further, we can bound the second term from above, by adding the summands with equal index $j_1 = j_2$. Then, $\Mb_n \le C (\Zb_{n, 1} + \Zb_{n, 2}^2)$, for some constant $C > 0$ and
\begin{equation*} 
	\Zb_{n, 2} =\frac{1}{n} \sum_{j=1}^{\ell_n} \sum_{k_1= 1}^{b_n} \sum_{k_2 = k_1 - m_n}^{k_1 + m_n} \ex[\tilde{\eps}_{k_1+(j_1-1)b_n, n} \tilde{\eps}_{k_2+(j_1-1)b_n, n} ] \prod_{i=1}^2 A_{k_i, j} B_{k_i, j}
\end{equation*}
The moments can be uniformly bounded, since $\sup_{1\le i\le n, n\in\N}\ex[\eps_{i, n}^4] < \infty$ by assumption. For $\Zb_{n, 1}$, by $m_n$ dependence and taking the range of $j$ into account, as specified by the indicators $B_{k_i, j}$, there are at most $C m_n^2 b_n^2 \big(\tfrac{|s_1 - s_2|n+b_n }{b_n}\big)$ non-zero summands , so that
\begin{equation} \label{eq:moment_order}
	\Zb_{n, 1} \le C \frac{m_n^2}{\ell_n} \big(|s_1 - s_2|+\tfrac{1}{\ell_n}\big).
\end{equation}
For all $s_1, s_2 \in [0, 1]$ with $ |s_1 - s_2| > \tfrac{1}{\ell_n}$, it holds
\[
	\Zb_{n, 1} \le C \sqrt{\frac{m_n^4}{\ell_n}} \frac{1}{\sqrt{\ell_n}}|s_1 - s_2| \le  C |s_1 - s_2|^{3/2}
\]
since $m_n^4 = \Oc(\ell_n)$.
To bound $\Zb_{n, 2}$, note that 
\[ 
B_{k, j} = \id\big( \tfrac{\lfloor (s_1 \wedge s_2) n \rfloor}{b_n} + 1 < j \le \tfrac{\lfloor (s_1 \vee s_2) n \rfloor}{b_n} + 1 \big) 
\]
for all $j\in \{1, \dots, \ell_n\} \setminus \{ \lfloor \tfrac{(s_1 \wedge s_2) n}{b_n}\rfloor + 2,  \lfloor \tfrac{(s_1 \vee s_2) n}{b_n}\rfloor + 1 \}$. For any such $j$, due to \eqref{eq:m_dependent_cov} and Proposition \ref{prop:lrv_approximation}, 
\begin{align}\label{eq:lrv_approx}
	& \sum_{k_1= 1}^{b_n} \sum_{k_2 = k_1 - m_n}^{k_1 + m_n} \ex[\tilde{\eps}_{k_1+(j-1)b_n, n} \tilde{\eps}_{k_2+(j-1)b_n, n} ] \prod_{i=1}^2 A_{k_i, j} B_{k_i, j} \\
	& = \frac{t n + \ell_n}{\ell_n} \big( \sigma^2\big(\tfrac{j}{\ell_n}\big) + o(1)\big)  \id\big( \tfrac{\lfloor (s_1 \wedge s_2) n \rfloor}{b_n} + 1 < j \le \tfrac{\lfloor (s_1 \vee s_2) n \rfloor}{b_n} + 1 \big), \notag
\end{align}
since $b_nm_n\Theta_{m_n} = o(b_n)$ and $\tfrac{b_n^2m_n}{n} = o(b_n)$.
Similarly, the quantity is of order $\Oc(b_n)$ at the boundaries, for $j \in \{ \lfloor \tfrac{(s_1 \wedge s_2) n}{b_n}\rfloor + 2,  \lfloor \tfrac{(s_1 \vee s_2) n}{b_n}\rfloor + 1 \}$. Accounting for the factor $\tfrac{1}{n}$, the contribution of the boundaries is of order $\Oc(b_n/n)$. Hence, $\Zb_{n, 2}$ can be rewritten as
\begin{align*}
	\Zb_{n, 2} & = \frac{1}{n} \sum_{j=1}^{\ell_n}\frac{t n + \ell_n}{\ell_n} \big( \sigma^2\big(\tfrac{j}{\ell_n}\big) + o(1)\big)  \id\big( \tfrac{\lfloor (s_1 \wedge s_2) n \rfloor}{b_n} + 1 < j \le \tfrac{\lfloor (s_1 \vee s_2) n \rfloor}{b_n} + 1 \big) + \Oc\big(\tfrac{1}{\ell_n}\big)\\
	& \le C \frac{1}{n}  \frac{|s_1 - s_2| n + b_n}{b_n} b_n + \Oc\big(\tfrac{1}{\ell_n}\big) \le C |s_1 - s_2| + \Oc\big(\tfrac{1}{\ell_n}\big).
\end{align*}
Hence, similarly to $\Zb_{n, 1}$, for all $s_1, s_2 \in [0, 1]$ such that  $|s_1 - s_2| > \tfrac{1}{\ell_n}$, it holds
\[
\Zb_{n, 2}^2 \le C |s_1 - s_2|^2.
\]
Combining the bounds for $\Zb_{n, 1}$ and $\Zb_{n, 2}$, we finally have
\[ 
	\Mb_n^{1/4} = \ex\big[\big(\tilde{G}_n(t, s_1) - \tilde{G}_n(t, s_2) \big)^4\big]^{1/4} \le C  |s_1 - s_2|^{3/8},
\] 
for all $s_1, s_2 \in [0, 1]$ with $|s_1 - s_2| > \tfrac{1}{2^{8/3}\ell_n}$. 

By Lemma A.1 of \cite{kley2016}, for any $\rho > 0, \eta \ge \tfrac{1}{2\ell_n^{3/8}}$, it holds
\begin{align} \label{eq:tightness3}
	& \ex\bigg[\sup_{t \in [0, 1], |s_1 - s_2| \le \rho} \big(\tilde{G}_n(t, s_1) - \tilde{G}_n(t, s_2)\big)^4\bigg]^{1/4}  \\
	& \le K \bigg\{ \int_{\tfrac{1}{2\ell_n^{3/8}}}^{\eta} D^{1/4}(\eps)d\eps +\big( \rho^{3/8} + \tfrac{2}{\ell_n^{3/8}} \big) D^{1/2}(\eta)  \bigg\} 
	 + 2 \ex\bigg[ \sup_{\substack{|s_1 - s_2| \le \ell_n^{-1}\\ t \in [0, 1], s_1 \in \Tb}} \big| \tilde{G}_n(t, s_1) - \tilde{G}_n(t, s_2) \big|^4  \bigg]^{1/4} \notag
\end{align}
for some constant $K$, where $D(\eps)$ denotes the packing number of the space $([0, 1], | \cdot |^{3/8})$ and $\Tb$ consists of at most $D(\ell_n^{-3/8})$ points. $D(\eps)$ can be bounded from above by $\eps^{-8/3}$, so that $D(\ell_n^{-3/8})\le \ell_n$. For the first summand, it holds
\[ 
	\int_{\tfrac{1}{2\ell_n^{3/8}}}^{\eta} D^{1/4}(\eps)d\eps +\big( \rho^{3/8} + 2 \ell_n^{-3/8} \big) D^{1/2}(\eta) \le 3 \eta^{1/3} - \tfrac{3}{2^{1/3}} \ell_n^{-1/8} + \big(\rho^{3/8} + 2 \ell_n^{-3/8} \big)\tfrac{1}{\eta^{4/3}}. 
\]
For the second summand, we can bound
\begin{align} \label{eq:tightness2}
	& \ex\bigg[ \sup_{\substack{|s_1 - s_2| \le \ell_n^{-1}\\ t \in [0, 1], s_1 \in \Tb}} \big| \tilde{G}_n(t, s_1) - \tilde{G}_n(t, s_2) \big|^4  \bigg] \\
	& \le \ex\bigg[ \sup_{\substack{s_2 - s_1 \in [0, \ell_n^{-1}],\\ t \in [0, 1],s_1 \in \Tb}} \big| \tilde{G}_n(t, s_1) - \tilde{G}_n(t, s_2) \big|^4  \bigg] + \ex\bigg[ \sup_{\substack{s_1 - s_2 \in [0, \ell_n^{-1}],\\ t \in [0, 1], s_1 \in \Tb}} \big| \tilde{G}_n(t, s_1) - \tilde{G}_n(t, s_2) \big|^4  \bigg]. \notag
\end{align}
For the first expectation on the right-hand side, we have 
\begin{align} \label{eq:tightness4}
	\Rb_n & := \ex\bigg[ \sup_{\substack{s_2 - s_1 \in [0, \ell_n^{-1}],\\ t \in [0, 1],s_1 \in \Tb}} \big| \tilde{G}_n(t, s_1) - \tilde{G}_n(t, s_2) \big|^4  \bigg] \\
	& \le \sum_{s \in \Tb} \ex\bigg[ \sup_{t\in[0, 1]} \max_{i=1}^{b_n} \bigg| \frac{1}{\sqrt{n}} \sum_{j=1}^{\ell_n} \sum_{k=1}^{b_n} \tilde{\eps}_{k+(j-1)b_n, n} \id\big( k \le \tfrac{\lfloor t  n \rfloor - j}{\ell_n} + 1  \big) \notag \\
	& \hspace{5.5cm} \times \id\big( \tfrac{\lfloor s n \rfloor - k}{b_n} + 1 < j \le \tfrac{\lfloor s n \rfloor +i - k}{b_n} + 1 \big)    \bigg|^4  \bigg]. \notag
\end{align}
Further note that the indicator
$\id\big( \tfrac{\lfloor s n \rfloor - k}{b_n} + 1 < j \le \tfrac{\lfloor s n \rfloor +i - k}{b_n} + 1 \big)$ is not $0$, only if $j = \lfloor \frac{sn + i - k}{b_n}\rfloor + 1$ and $\lfloor \frac{sn + i - k}{b_n}\rfloor > \lfloor \frac{sn - k}{b_n}\rfloor$. Hence, for each $k$, at most one $j(k)$ exists, such that the indicator does not vanish. In this case, $j(k) = \lfloor \frac{sn + i - k}{b_n}\rfloor + 1$. 

In particular, it follows from \eqref{eq:tightness4} that
\begin{equation} 
	\Rb_n \le \sum_{s \in \Tb} \ex\bigg[ \sup_{t\in[0, 1]} \max_{i=1}^{b_n} \bigg| \frac{1}{\sqrt{n}} \sum_{k=1}^{b_n} \tilde{\eps}_{k+(j(k)-1)b_n, n} \id\big( k \le \tfrac{\lfloor t  n \rfloor - j(k)}{\ell_n} + 1 , \lfloor \tfrac{sn - k}{b_n} \rfloor < \lfloor \tfrac{sn - k + i}{b_n} \rfloor \big)  \bigg|^4  \bigg] \label{eq:tightness5}
\end{equation}
Let $r_s = \lfloor sn \rfloor - \lfloor \tfrac{sn}{b_n}\rfloor b_n$, then, $\lfloor \frac{sn + i - k}{b_n}\rfloor > \lfloor \frac{sn - k}{b_n}\rfloor$ if and only if $r_s < k \le r_s + i$ or $k \le r_s + i - b_n$. Hence, we may rewrite
\begin{align*}
	& \sup_{t\in[0, 1]} \max_{i=1}^{b_n} \bigg| \frac{1}{\sqrt{n}} \sum_{k=1}^{b_n} \tilde{\eps}_{k+(j(k)-1)b_n, n} \id\big( k \le \tfrac{\lfloor t  n \rfloor - j(k)}{\ell_n} + 1,r_s < k \le r_s + i \vee k \le r_s + i -b_n \big) \bigg| \\
	& \le \sup_{t\in[0, 1]} \max_{i=1}^{b_n} \bigg| \frac{1}{\sqrt{n}} \sum_{k=1}^{b_n} \tilde{\eps}_{k+(j(k)-1)b_n, n} \id\big(r_s < k \le \min(\tfrac{\lfloor t  n \rfloor - j(k)}{\ell_n} +1, r_s + i)\big)\bigg|  \\
	& ~ + \sup_{t\in[0, 1]} \max_{i=1}^{b_n} \bigg| \frac{1}{\sqrt{n}} \sum_{k=1}^{b_n} \tilde{\eps}_{k+(j(k)-1)b_n, n} \id\big(k \le \min(\tfrac{\lfloor t  n \rfloor - j(k)}{\ell_n} + 1, r_s + i -b_n)\big) \bigg| \\
	& \le \max_{\nu=r_s+1}^{b_n} \bigg| \frac{1}{\sqrt{n}} \sum_{k=r_s+1}^{\nu} \tilde{\eps}_{k+(j(k)-1)b_n, n} \bigg| + \max_{\nu=1}^{b_n} \bigg| \frac{1}{\sqrt{n}} \sum_{k=1}^{\nu} \tilde{\eps}_{k+(j(k)-1)b_n, n} \bigg| 
\end{align*}
By \eqref{eq:tightness5}, $\Rb_n$ can be bounded from above by
\begin{equation*} 
	C \sum_{s \in \Tb} \bigg( \sum_{\nu=r_s+1}^{b_n} \ex\bigg[\bigg| \frac{1}{\sqrt{n}} \sum_{k=r_s+1}^{\nu} \tilde{\eps}_{k+(j(k)-1)b_n, n} \bigg|^4\bigg] + \sum_{\nu=1}^{b_n} \ex\bigg[\bigg| \frac{1}{\sqrt{n}} \sum_{k=1}^{\nu} \tilde{\eps}_{k+(j(k)-1)b_n, n} \bigg|^4\bigg]  \bigg),
\end{equation*}
where the expectations are of order $\Oc(\tfrac{b_n^2 m_n^2}{n^2})$ by the same arguments that led to \eqref{eq:moment_order}. Since $|\Tb| \le D(\ell_n^{-3/8}) \le \ell_n$, $\Rb_n$ is of order $\Oc(\tfrac{b_n^2 m_n^2}{n})$. 

Analogously, we can bound the second expression on the right-hand side of \eqref{eq:tightness2}, so that
\[ \ex\bigg[ \sup_{\substack{|s_1 - s_2| \le \ell_n^{-1}\\ t \in [0, 1], s_1 \in \Tb}} \big| \tilde{G}_n(t, s_1) - \tilde{G}_n(t, s_2) \big|^4  \bigg] \le C \tfrac{b_n^2 m_n^2}{n},  \]
for some generic constant $C\in\R$. By Markov's inequality and \eqref{eq:tightness3}, it follows 
\begin{align*}
	\lim_{\rho \searrow 0} \lim_{n\to\infty} \pr \Big( \sup_{t \in [0, 1], |s_1 - s_2| \le \rho} |G_n(t, s_1) - G_n(t, s_2)| > \eps \Big) \le \frac{81 K^4 \eta^{4/3}}{\eps^4},     
\end{align*}
for any $\eta > 0$, which completes the proof of \eqref{eq:tight1}.

The proof of \eqref{eq:tight2}, generally follows by similar arguments. As before, for $t_1, t_2, s \in[0, 1]$, we can rewrite 
\begin{align*}
	\tilde{G}_n(t_1, s) - \tilde{G}_n(t_2, s) 
	= \operatorname{sign}(t_1 - t_2) \frac{1}{\sqrt{n}} \sum_{j=1}^{\ell_n} \sum_{k=1}^{b_n} & \tilde{\eps}_{k+(j-1)b_n, n} \id(j \le \tfrac{\lfloor s n \rfloor - k}{b_n} + 1) \\
	& \times \id\Big(\tfrac{\lfloor (t_1 \wedge t_2) n \rfloor - j}{\ell_n} + 1 < k \le \tfrac{\lfloor (t_1 \vee t_2) n \rfloor - j}{\ell_n} + 1\Big).
\end{align*}
For $\tilde{A}_{k, j} = \id(j \le \tfrac{\lfloor s n \rfloor - k}{b_n} + 1)$ and $\tilde{B}_{k, j} = \id\big(\tfrac{\lfloor (t_1 \wedge t_2) n \rfloor - j}{\ell_n} + 1 < k \le \tfrac{\lfloor (t_1 \vee t_2) n \rfloor - j}{\ell_n} + 1\big)$, we can bound
\[ \tilde{\Mb}_n := \ex\big[\big(\tilde{G}_n(t_1, s) - \tilde{G}_n(t_2, s) \big)^4\big] \le C (\Zb_{n, 3} + \Zb_{n, 4}^2), \]
for some generic constant $C\in\R$, 
\[ 	
	\Zb_{n, 3} = \frac{1}{n^2} \sum_{j=1}^{\ell_n} \sum_{k_1=1}^{b_n} \sum_{k_2=1-m_n}^{b_n+m_n} \sum_{k_3=1-2m_n}^{b_n+2m_n} \sum_{k_4=1-3m_n}^{b_n+3m_n} \ex\bigg[\prod_{i=1}^4 \tilde{\eps}_{k_i+(j-1)b_n, n} \bigg]  \prod_{i=1}^4 \tilde{A}_{k_i, j} \tilde{B}_{k_i, j}
\]
and 
\[ 	
	\Zb_{n, 4} =\frac{1}{n} \sum_{j=1}^{\ell_n} \sum_{k_1= 1}^{b_n} \sum_{k_2 = k_1 - m_n}^{k_1 + m_n} \ex[\tilde{\eps}_{k_1+(j_1-1)b_n, n} \tilde{\eps}_{k_2+(j_1-1)b_n, n} ] \prod_{i=1}^2 \tilde{A}_{k_i, j} \tilde{B}_{k_i, j}.
\]

By taking the ranges of $j$ and $k$ into account, as specified by the indicators, and $m_n$ dependence, there are at most $C \ell_n m_n^2 \big(\tfrac{|t_1 - t_2|n+\ell_n}{\ell_n}\big)^2$ non-zero summands, so that, similarly to \eqref{eq:moment_order},
\begin{equation*}
	\Zb_{n, 3} \le C \frac{m_n^2}{\ell_n} \big(|t_1 - t_2| + \tfrac{1}{b_n}\big)^2.
\end{equation*}
For all $t_1, t_2 \in [0, 1]$ with $ |t_1 - t_2| > \tfrac{1}{b_n}$, it holds
$
\Zb_{n, 3} \le  C |t_1 - t_2|^2
$
since $m_n^2 = \Oc(\ell_n)$. To bound $\Zb_{n, 4}$, note that 
\[ 
\tilde{A}_{k, j} = \id(j \le \tfrac{\lfloor s n \rfloor}{b_n} + 1)
\]
for all $j\in \{1, \dots, \ell_n\} \setminus \{ \lfloor \tfrac{s n}{b_n}\rfloor + 1 \}$. Analogously to \eqref{eq:lrv_approx}, for any such $j$,
\begin{align*}
	& \sum_{k_1= 1}^{b_n} \sum_{k_2 = k_1 - m_n}^{k_1 + m_n} \ex[\tilde{\eps}_{k_1+(j_1-1)b_n, n} \tilde{\eps}_{k_2+(j_1-1)b_n, n} ] \prod_{i=1}^2 A_{k_i, j} B_{k_i, j} \\
	& = \frac{|t_1 - t_2| n + \ell_n}{\ell_n} \big( \sigma^2\big(\tfrac{j}{\ell_n}\big) + o(1)\big)  \id\big( j \le \tfrac{\lfloor s n \rfloor}{b_n} + 1 \big),
\end{align*}
and the quantity is of order $\Oc(b_n)$ at $j= \lfloor \tfrac{s n}{b_n}\rfloor + 1$. Hence, we can rewrite $\Zb_{n, 4}$ as 
\begin{align*}
	\Zb_{n, 4} & = \frac{1}{n} \sum_{j=1}^{\ell_n}\frac{|t_1 - t_2| n + \ell_n}{\ell_n} \big( \sigma^2\big(\tfrac{j}{\ell_n}\big) + o(1)\big)  \id\big( j \le \tfrac{\lfloor s n \rfloor}{b_n} + 1 \big) + \Oc\big(\tfrac{1}{\ell_n}\big)\\
	& 
	\le C |t_1 - t_2| + \Oc\big(\tfrac{1}{b_n}\big).
\end{align*}
Similarly to $\Zb_{n, 3}$, for all $t_1, t_2 \in [0, 1]$ such that  $|t_1 - t_2| > \tfrac{1}{b_n}$, it holds
\[
\Zb_{n, 4}^2 \le C |t_1 - t_2|^2.
\]
Combining the bounds for $\Zb_{n, 3}$ and $\Zb_{n, 4}$, we finally have
\[ 
\tilde{\Mb}_n^{1/4} = \ex\big[\big(\tilde{G}_n(t_1, s) - \tilde{G}_n(t_2, s) \big)^4\big]^{1/4} \le C  |t_1 - t_2|^{1/2},
\] 
for all $t_1, t_2 \in [0, 1]$ with $|t_1 - t_2| > \tfrac{1}{4b_n}$. 

By Lemma A.1 of \cite{kley2016}, for any $\rho > 0, \eta \ge \tfrac{1}{\sqrt{b_n}}$, it holds
\begin{align} \label{eq:tightness9}
	& \ex\bigg[\sup_{s \in [0, 1], |t_1 - t_2| \le \rho} \big(\tilde{G}_n(t_1, s) - \tilde{G}_n(t_2, s)\big)^4\bigg]^{1/4}  \\
	& \le K \bigg\{ \int_{\tfrac{1}{2\sqrt{b_n}}}^{\eta} D^{1/4}(\eps)d\eps +\big( \rho^{1/2} + \tfrac{2}{\sqrt{b_n}} \big) D^{1/2}(\eta)  \bigg\} 
	+ 2 \ex\bigg[ \sup_{\substack{|t_1 - t_2| \le b_n^{-1}\\ s \in [0, 1], t_1 \in \tilde{\Tb}}} \big| \tilde{G}_n(t_1, s) - \tilde{G}_n(t_2, s) \big|^4  \bigg]^{1/4} \notag
\end{align}
for some constant $K$, where $D(\eps)$ denotes the packing number of the space $([0, 1], | \cdot |^{1/2})$ and $\Tb$ consists of at most $D(b_n^{-1/2})$ points. $D(\eps)$ can be bounded from above by $\eps^{-2}$, so that $D(b_n^{-1/2})\le b_n$. The first summand can be bounded by
$
K \Big( 2 \sqrt{\eta} - b_n^{-1/4} + \big(\sqrt{\rho} + \tfrac{2}{\sqrt{b_n}} \big)\tfrac{1}{\eta}\Big). 
$
As before, we split the second summand
\begin{align} \label{eq:tightness8}
	& \ex\bigg[ \sup_{\substack{|t_1 - t_2| \le b_n^{-1}\\ s \in [0, 1], t_1 \in \tilde{\Tb}}} \big| \tilde{G}_n(t_1, s) - \tilde{G}_n(t_2, s) \big|^4  \bigg] \\
	& \le \ex\bigg[ \sup_{\substack{t_2 - t_1 \in [0, b_n^{-1}]\\ s \in [0, 1], t_1 \in \tilde{\Tb}}} \big| \tilde{G}_n(t_1, s) - \tilde{G}_n(t_2, s) \big|^4  \bigg] + \ex\bigg[ \sup_{\substack{t_1 - t_2 \in [0, b_n^{-1}]\\ s \in [0, 1], t_1 \in \tilde{\Tb}}} \big| \tilde{G}_n(t_1, s) - \tilde{G}_n(t_2, s) \big|^4  \bigg]. \notag
\end{align}
We can bound the first expectation on the right-hand side by
\begin{align*}
	\tilde{\Rb}_n & := \ex\bigg[ \sup_{\substack{t_2 - t_1 \in [0, b_n^{-1}]\\ s \in [0, 1], t_1 \in \tilde{\Tb}}} \big| \tilde{G}_n(t_1, s) - \tilde{G}_n(t_2, s) \big|^4  \bigg] \\
	& \le \sum_{t \in \Tb} \ex\bigg[ \sup_{s\in[0, 1]} \max_{i=1}^{\ell_n} \bigg| \frac{1}{\sqrt{n}} \sum_{j=1}^{\ell_n} \sum_{k=1}^{b_n} \tilde{\eps}_{k+(j-1)b_n, n} \id\big(j \le \tfrac{\lfloor s  n \rfloor - k}{b_n} + 1  \big) \notag \\
	& \hspace{7cm} \times \id\big( \tfrac{\lfloor t n \rfloor - j}{\ell_n} + 1 < k \le \tfrac{\lfloor t n \rfloor + i - j}{\ell_n} + 1 \big)    \bigg|^4  \bigg]. \notag
\end{align*}
Note that the indicator $\id\big( \tfrac{\lfloor t n \rfloor - j}{\ell_n} + 1 < k \le \tfrac{\lfloor t n \rfloor + i - j}{\ell_n} + 1 \big)$ is only non-zero if $k =  \lfloor \tfrac{\lfloor t n \rfloor - j}{\ell_n}\rfloor + 2 \le  \tfrac{\lfloor t n \rfloor + i - j}{\ell_n} + 1$. Hence, for each $j$ at most one summand with index $k(j)$ exists, so that 
\begin{equation} \label{eq:tightness7}
	\tilde{\Rb}_n \le \sum_{t \in \Tb} \ex\Big[ \sup_{s\in[0, 1]} \max_{i=1}^{\ell_n} \big| M_n( \tfrac{\lfloor s  n \rfloor}{b_n} + 1, i)
	\big|^4  \Big], 
\end{equation}
where 
\[
M_n(x, i) = \frac{1}{\sqrt{n}} \sum_{j=1}^{\ell_n}  \tilde{\eps}_{k(j)+(j-1)b_n, n} \id\big(j \le x - \tfrac{k(j)}{b_n}  \big) \id\big(\lfloor \tfrac{\lfloor t n \rfloor - j}{\ell_n}\rfloor + 2 \le \tfrac{\lfloor t n \rfloor + i - j}{\ell_n} + 1 \big). 
\]
The supremum over $s$, on the right-hand side of \eqref{eq:tightness7}, can be replaced by a discrete maximum, so that
\[
\sum_{t \in \Tb} \ex\Big[ \sup_{s\in[0, 1]} \max_{i=1}^{\ell_n} \big| M_n( \tfrac{\lfloor s  n \rfloor}{b_n} + 1, i)
\big|^4  \Big] = \sum_{t \in \Tb} \ex\big[ \max_{\nu=1}^{\ell_n} \max_{i=1}^{\ell_n} | M_n(\nu, i)|^4  \big].
\]
Since we have at most one term $\tilde{\eps}_{k(j)+(j-1)b_n, n}$ for each $j \in \{1, \dots, \ell_n\}$, and the distance between two terms is approximately $b_n$, the random variables are independent, due to their $m_n$-dependence. The indicators  $\id\big(j \le \nu  \big)$ and $\id\big(\lfloor \tfrac{\lfloor t n \rfloor - j}{\ell_n}\rfloor + 2 \le \tfrac{\lfloor t n \rfloor + i - j}{\ell_n} + 1 \big)$ are increasing in $\nu$ and $i$, respectively, and the random variables are centered and independent. Hence, if we fix one index ($\nu$ or $i$), $M_n(\nu, i)$ is a martingale with respect to the other index. Therefore, $M_n(\nu, i)$ is an orthosubmartingale and we can apply Cairoli's maximal inequality \citep[see, e.\,g., Theorem 2.3.1 in][]{khoshnevisan2006} to bound
\begin{equation*}
	\sum_{t \in \Tb} \ex\big[ \max_{\nu=1}^{\ell_n} \max_{i=1}^{\ell_n} | M_n(\nu, i)|^4  \big] 
	\le \sum_{t \in \Tb} \ex\big[ | M_n(\ell_n, \ell_n)|^4  \big] 
	= \sum_{t \in \Tb} \frac{1}{n^2} \ex\bigg[ \bigg| \sum_{j=1}^{\ell_n}  \tilde{\eps}_{k(j)+(j-1)b_n, n} \bigg|^4 \bigg]. 
\end{equation*}
By the same arguments that led to bounds for $\Mb_n$ and $\tilde{\Mb}_n$, the expectation on the right-hand side is of order $\Oc(\ell_n^2 m_n^2)$, so that
\[ \sum_{t \in \Tb} \frac{1}{n^2} \ex\bigg[ \bigg| \sum_{j=1}^{\ell_n}  \tilde{\eps}_{k(j)+(j-1)b_n, n} \bigg|^4 \bigg] \le C \frac{b_n}{n^2} \ell_n^2 m_n^2 = C \frac{m_n^2}{b_n}. \]
%
We can bound the second expectation in \eqref{eq:tightness8} analogously. By Markov's inequality and \eqref{eq:tightness9}, it follows 
\begin{align*}
	\lim_{\rho \searrow 0} \lim_{n\to\infty} \pr \Big( \sup_{s \in [0, 1], |t_1 - t_2| \le \rho} |G_n(t_1, s) - G_n(t_2, s)| > \eps \Big) \le \frac{16 K^4 \eta^2}{\eps^4},     
\end{align*}
for any $\eta > 0$, which proves \eqref{eq:tight2} and completes the proof of the lemma.

\subsection{Proof of Results from Section \ref{sec:change_detection}} \label{sec:proof_statistics}

\begin{proposition} \label{prop:mu_approx}
	Let Assumption \ref{assump:mu} be satisfied. Then,
	\[ \ex[S_n(t, s)] = \frac{\lfloor \tfrac{n t}{\ell_n}\rfloor}{b_n} \int_0^s \mu(x) \diff x - \frac{1}{b_n} \int_{s \wedge (\lfloor nt \rfloor - \lfloor \tfrac{n t}{\ell_n}\rfloor \ell_n)\tfrac{b_n}{n}}^s \mu(x) \diff x + \Oc\big(\tfrac{b_n}{n}\big), \]
	uniformly for $t, s \in[0, 1]$.
\end{proposition}

\begin{proof}
	Without loss of generality, assume that $\mu$ is Lipschitz continuous. If $\mu$ is only piecewise Lipschitz continuous, a finite number of jump points $p$ exists, and the following arguments can be used for each segment between jump points separately. Note, that
	\begin{align}
			\ex[S_n(t, s)] & = \frac{1}{n} \sum_{i=1}^{\ell_n b_n} \mu\big(\tfrac{\pi_i}{n}\big) \id( i \le \lfloor tn \rfloor, \pi_i \le \lfloor sn \rfloor) + \Oc\big(\tfrac{b_n}{n}\big) \notag \\
			& = \frac{1}{n} \sum_{j=1}^{\ell_n} \sum_{k=1}^{b_n} \mu\big(\tfrac{k+(j-1)b_n}{n}\big) \id( k \le \tfrac{tn - j}{\ell_n}, j \le \tfrac{sn-k}{b_n} ) + \Oc\big(\tfrac{b_n}{n}\big) \label{eq:approx_mu}\\
			& = \frac{1}{n} \sum_{j=1}^{\ell_n} \id(j \le \tfrac{sn}{b_n})\sum_{k=1}^{b_n} \mu\big(\tfrac{k+(j-1)b_n}{n}\big) \id( k \le \tfrac{tn - j}{\ell_n}) + \Oc\big(\tfrac{b_n}{n}\big), \notag
		\end{align}
	where the last equality follows, since at most one $j^* \in \{1, \dots, \ell_n\}$ exists such that $\id(j \le \tfrac{sn}{b_n}) \neq \id(j \le \tfrac{sn-k}{b_n})$. By Lipschitz continuity of $\mu$, 
	\[ \mu\big(\tfrac{k+(j-1)b_n}{n}\big) = \frac{n}{b_n} \int_{(j-1)b_n/n}^{jb_n/n} \mu(x) \diff x + \Oc\big(\tfrac{b_n}{n}\big), \]
	uniformly in $k = 1, \dots, b_n$ and $j=1,\dots, \ell_n$. Therefore, the right-hand side of \eqref{eq:approx_mu} can be rewritten as
	\begin{equation} \label{eq:approx_mu2}
		\frac{1}{b_n} \sum_{k=1}^{b_n} \sum_{j=1}^{\ell_n} \id(j \le \tfrac{sn}{b_n}) \id( k \le \tfrac{tn - j}{\ell_n}) \int_{(j-1)b_n/n}^{jb_n/n} \mu(x) \diff x + \Oc\big(\tfrac{b_n}{n}\big).
	\end{equation}
	For $k > \lfloor \tfrac{tn}{\ell_n}\rfloor =: k^*$,  $\id( k \le \tfrac{tn - j}{\ell_n}) = 0$, whereas the indicator is $1$ for $k < k^*$. For the boundary $k^*$, it holds 
	\[
		\id(j \le \tfrac{sn}{b_n}, j \le tn - k^* \ell_n) = \id(j \le \tfrac{sn}{b_n}) - \id(j \le \tfrac{sn}{b_n}, j > tn - k^* \ell_n),
	\]
	so that
	\begin{align*}
		& \frac{1}{b_n} \sum_{j=1}^{\ell_n} \id(j \le \tfrac{sn}{b_n}, j \le tn - k^* \ell_n) \int_{\tfrac{(j-1)b_n}{n}}^{\tfrac{jb_n}{n}} \mu(x) \diff x \\
		& = \frac{1}{b_n} \int_{0}^{\lfloor \tfrac{sn}{b_n} \rfloor \tfrac{b_n}{n}} \mu(x) \diff x - \frac{1}{b_n} \int_{(\lfloor (tn - k^* \ell_n) \wedge \tfrac{sn}{b_n} \rfloor) \tfrac{b_n}{n}}^{\lfloor \tfrac{sn}{b_n} \rfloor \tfrac{b_n}{n}} \mu(x) \diff x.
	\end{align*}	
	Therefore, \eqref{eq:approx_mu2} can be simplified to 
	\begin{align*}
		& \frac{k^*}{b_n}  \int_{0}^{\lfloor \tfrac{sn}{b_n} \rfloor \tfrac{b_n}{n}} \mu(x) \diff x - \frac{1}{b_n} \int_{(\lfloor (tn - k^* \ell_n) \wedge \tfrac{sn}{b_n} \rfloor) \tfrac{b_n}{n}}^{\lfloor \tfrac{sn}{b_n} \rfloor \tfrac{b_n}{n}} \mu(x) \diff x +  \Oc\big(\tfrac{b_n}{n}\big).
	\end{align*}
	Finally, the proposition follows by replacing $\lfloor \tfrac{sn}{b_n} \rfloor \tfrac{b_n}{n}$ with $s$, which yields an additional error term of order $\Oc\big(\tfrac{b_n}{n}\big)$.
\end{proof}

\begin{proposition} \label{prop:mu_const}
	Let Assumption \ref{assump:mu} be satisfied. Then,
	\begin{equation} \label{eq:mu_const}
		s \mu(s) = \int_0^s \mu(x) \diff x
	\end{equation}
	for all $s\in [0, 1]$, if and only if $\mu$ is constant.
\end{proposition}

\begin{proof}
	If $\mu$ is constant with $\mu(s) = c$, 
	\[ s \mu(s) = s c = \int_0^s c \diff x = \int_0^s \mu(x) \diff x. \]
	
	Contrarily, let \eqref{eq:mu_const} be true for all $s\in [0, 1]$. First assume that $\mu$ has a jump point in $s_0 \in (0, 1)$. Then, for any $\delta > 0 $, 
	\begin{align*}
		(s_0 + \delta) \mu(s_0 + \delta) - (s_0 - \delta) \mu(s_0 - \delta) 
	= \int_0^{s_0 + \delta} \mu(x) \diff x - \int_0^{s_0 - \delta} \mu(x) \diff x 
		= \int_{s_0 - \delta}^{s_0 + \delta} \mu(x) \diff x.
	\end{align*}
	Since $\mu$ is piecewise Lipschitz continuous on $[0, 1]$, it is bounded, so that the right-hand side converges to $0$, for $\delta \to 0$. In particular, $\lim_{\delta \to 0} \mu(s_0 + \delta) = \lim_{\delta \to 0} \mu(s_0 - \delta)$, which is a contradiction because $s_0$ was assumed to be a jump point. Hence, $\mu$ does not have jump points and is Lipschitz continuous.
	
	By continuity, $\mu$ attains a maximum and minimum. Let
	\[ s_- = \min\{ s\in[0, 1]: \mu(s) = \min_{t \in [0, 1]} \mu(t) \} \quad \mathrm{and}\quad  s_+ = \min\{ s\in[0, 1]: \mu(s) = \max_{t \in [0, 1]} \mu(t) \}. \]
	By continuity, $s_-$ and $s_+$ are well-defined. Assume that $s_+ > 0$. By \eqref{eq:mu_const},
	 \[ \int_0^{s_+} \mu(s_+) - \mu(x) \diff x = s_+ \mu(s_+) - \int_0^{s_+} \mu(x)  \diff x = s_+ \mu(s_+) - s_+ \mu(s_+) = 0. \]
	Since $\mu(s_+) - \mu(x)$ is a non-negative function, it must be equal to $0$, so that $\mu(x) = \mu(s_+)$ for $x\in [0, s_+]$. This contradicts the definition of $s_+$, hence $s_+ = 0$. By the same arguments $s_- = 0$, so that 
	$\min_{s \in [0, 1]} \mu(s) = \max_{s \in [0, 1]} \mu(s)$
	and $\mu$ is constant.
\end{proof}

\noindent  \textbf{Proof of Corollary \ref{cor:simple}}. 

\noindent First note that 
\begin{align*}
	\sup_{t, s\in[0, 1]} \Big|\sqrt{n}\big( \tilde{S}_n(t, s) - \ex[\tilde{S}_n(t, s)] \big) - G_n(t, s) \Big|
	& = \sup_{t, s\in[0, 1]} \Big| G_n(\lfloor \tfrac{tn}{\ell_n}\rfloor \tfrac{\ell_n}{n}, s) - G_n(t, s) \Big| \\
	& \le \sup_{\substack{|t_1 - t_2| \le \Delta\\ s\in[0, 1]}} \Big| \sqrt{n} \big( G_n(t_1, s) - G_n(t_2, s) \big) \Big|,
\end{align*}
where $\Delta_n = \sup_{t\in[0, 1]}|t - \lfloor \tfrac{tn}{\ell_n}\rfloor \tfrac{\ell_n}{n}| \le \tfrac{\ell_n}{n}$. By Lemma \ref{lem:equicont} and Slutsky's theorem, it follows that $\{\sqrt{n}\big( \tilde{S}_n(t, s) - \ex[\tilde{S}_n(t, s)] \big)\}_{t, s\in[0, 1]}$ converges weakly to $G$.

 Let $B^{(1)}$ and $B^{(2)}$ denote independent Brownian motions. By Proposition \ref{prop:mu_approx}, $\sqrt{n}\big(\ex[\tilde{S}_n(t, 1)] - t_n \ex[\tilde{S}_n(1, 1)]\big) = o(1),$ uniformly in $t$. Hence, 
\[	
	\sqrt{n} \sup_{t\in[0, 1]} |\tilde{S}_n(t, 1) - t_n \tilde{S}_n(1, 1)|
	\convw \| \sigma \| \sup_{t\in[0, 1]}	 \big|B^{(2)}(t) - t B^{(2)}(1) \big|.
\]
Under the null hypothesis, $\sqrt{n} S_n(1, s) = \sqrt{n} \big(S_n(1, s) - \ex[S_n(1, s)]\big)$, which converges weakly to $G(1, s)$, as a process in $s$, by Theorem \ref{thm:functional_clt}. By \eqref{eq:timeshift}, 
\[ 	\sup_{s\in[0, 1]}|G^{(1)}(s)| 
\stackrel{\Dc}{=} \sup_{s\in[0, 1]}|\| \sigma \| B^{(1)}(s)|, \notag \]
so that, 
\[
	\frac{\sqrt{n} \sup_{s\in[0, 1]}|S_n(1, s)|}{\sqrt{n} \sup_{t\in[0, 1]} |\tilde{S}_n(t, 1) - t_n \tilde{S}_n(1, 1)|} 
	\convw  \frac{\sup_{s\in[0, 1]}| B^{(1)}(s)|}{\sup_{t\in[0, 1]}	\big|B^{(2)}(t) - t B^{(2)}(1) \big|},
\]
as in \eqref{eq:conv_simple_null}. Contrarily, under $\tilde{H}_1$, 
\[ \sqrt{n} \sup_{s\in[0, 1]}|S_n(1, s)| \ge \sqrt{n} \sup_{s\in[0, 1]}|\ex[S_n(1, s)]| -  \sup_{s\in[0, 1]}|G_n(1, s)| , \]
which diverges to $\infty$, by Proposition \ref{prop:mu_approx} and Theorem \ref{thm:functional_clt}.

\noindent \textbf{Proof of Corollary \ref{cor:full}}.

\noindent  By Proposition \ref{prop:mu_approx}, it holds
\[
\ex[V_n(s)] = \sqrt{n}  \frac{\lfloor \tfrac{t_0 n}{\ell_n}\rfloor-1}{b_n} \int_0^s \bigg(\int_0^x \mu(z) \diff z - \frac{x}{s} \int_0^s \mu(z) \diff z \bigg) \diff x  + \Oc\big(\tfrac{b_n}{\sqrt{n}}\big).
\]
Further, define $f(s, x) = \int_0^x \mu(z) \diff z - \frac{x}{s} \int_0^s \mu(z) \diff z$ and $g(s):= \int_0^s f(s, x) \diff x$. Since $g(0)=0$, $g(s) = 0$, for all $s \in[0, 1]$, if and only if $g'(s)=0$. By the Leibniz integral rule and integration by parts,
\begin{align*}
	g'(s) & = f(s, s) + \int_0^s \frac{\partial f}{\partial s}(s, x) \diff x \\
	 &= \int_0^s \bigg( - \frac{x}{s} \mu(s) + \frac{x}{s^2} \int_0^s \mu(z) \diff z \bigg)\diff x 
	 = -\frac{s}{2} \mu(s) + \frac{1}{2} \int_0^s \mu(x) \diff x, 
\end{align*}
which equals $0$, if and only if $s \mu(s) = \int_0^s \mu(x) \diff x$ for all $s \in [0, 1]$. By Proposition \ref{prop:mu_const}, this is equivalent to $\mu$ being constant. Hence, $\ex[V_n(s)] = 0$ for all $s \in [0, 1]$ if and only if $\mu(s) = c$. Contrarily, for all $s \in [0,1]$ with $g(s) \neq 0$, $\lim_{n\to\infty} |\ex[V_n(s)]| = \infty$. 

By Theorem \ref{thm:functional_clt}, $V_n(s) - \ex[V_n(s)]$ converges weakly as a process to
\[
	\int_0^s G(t_0, x) - \frac{x}{s} G(t_0, s) \diff x = \int_0^s \int_0^{t_0} \int_0^s \Big[\id(z \le x) - \frac{x}{s}\Big] \sigma(z) \diff B(y, z) \diff x. 
\]
Since $\sigma(z) \equiv \sigma$ is constant and the integrand is deterministic and bounded, by the Fubini theorem for stochastic integrals, the right-hand side can be rewritten as
\[ \int_0^{t_0} \int_0^s \int_0^s \Big[\id(z \le x) - \frac{x}{s}\Big] \diff x\, \sigma(z) \diff B(y, z)  = \sigma \int_0^{t_0} \int_0^s \Big(\frac{s}{2}-z\Big) \diff B(y, z) =: V(s). \]
For a Brownian motion $B^{(1)}$, define the centered Gaussian process $\tilde{V}$ by
\begin{equation*}
	\tilde{V}(s) = \sigma \sqrt{t_0} B^{(1)}\bigg( \int_0^s \Big(\frac{s}{2}-z\Big)^2  \diff z \bigg).
\end{equation*}
Then,
\begin{align*}
	\cov(V(s_1), V(s_2)) 
	& = \sigma^2 t_0 \int_0^1 \id(z \le s_1, z \le s_2) \Big(\frac{s_1}{2}-z\Big) \Big(\frac{s_2}{2}-z\Big) \diff z \\
	& = \sigma^2 t_0 \tfrac{1}{12} (s_1 \wedge s_2)^3 \\
	& = \sigma^2 t_0 \min\Big\{ \int_0^{s_1} (s_1/2-z)^2 \diff z, \int_0^{s_2} (s_2/2-z)^2 \diff z \Big\} \\
	& = \cov(\tilde{V}(s_1), \tilde{V}(s_2)),
\end{align*}
such that $V$ and $\tilde{V}$ have the same distribution.
%
Regarding  the denominator of the test statistic, by Proposition \ref{prop:mu_approx}, 
\[ \frac{1}{\sqrt{n}} \ex[\tilde{H}_n(s)] = \frac{\lfloor \frac{t_1 n}{\ell_n}\rfloor - \lfloor \frac{t_0 n}{\ell_n}\rfloor}{b_n} \int_0^s \mu(x) \diff x - \frac{\lfloor \frac{t_1 n}{\ell_n} \rfloor - \lfloor \frac{t_0 n}{\ell_n} \rfloor}{\lfloor \frac{n}{\ell_n} \rfloor - \lfloor \frac{t_0 n}{\ell_n} \rfloor} \cdot \frac{\lfloor \frac{n}{\ell_n}\rfloor - \lfloor \frac{t_0 n}{\ell_n}\rfloor}{b_n} \int_0^s \mu(x) \diff x + \Oc\big(\tfrac{b_n}{n}\big), \]
where the two terms cancel, such that $\ex[\tilde{H}_n(s)] = \Oc\big(\tfrac{b_n}{\sqrt{n}}\big)$. 
By Theorem \ref{thm:functional_clt}, $\tilde{H}_n(s)$ converges weakly, as a process in $s$, to
\begin{align*}
	& G(t_1, s) - G(t_0, s) - \frac{t_1 - t_0}{1 - t_0} [G(1, s) - G(t_0, s)] \\
	& = \int_{t_0}^1 \int_0^1 \Big(\id(y \le t_1) - \frac{t_1 - t_0}{1 - t_0}\Big) \id(z \le s) \sigma(z) \diff B(y, z)=: \tilde{H}(s).
\end{align*}
In particular, $\ex[H_n(s)] = \Oc\big(\tfrac{b_n}{\sqrt{n}}\big)$, uniformly for all $s \in [0, 1]$, and $H_n(s) \convw \int_0^s \tilde{H}(x) - \tfrac{x}{s} \tilde{H}(s) \diff x=:H(s)$. Again, since $\sigma(z)\equiv \sigma$ is constant and by the Fubini theorem for stochastic integrals,
\begin{align*}
	H(s) & = \int_0^s\int_{t_0}^1 \int_0^1 \Big(\id(y \le t_1) - \frac{t_1 - t_0}{1 - t_0}\Big) \Big(\id(z \le x) - \frac{x}{s} \id(z \le s)\Big) \sigma(z) \diff B(y, z) \diff x \\
	& = \sigma \int_{t_0}^1 \int_0^s \Big(\frac{s}{2} - z\Big) \Big(\id(y \le t_1) - \frac{t_1 - t_0}{1 - t_0}\Big) \diff B(y, z).
\end{align*}
Note that in the definition of $V$, we integrate with respect to $y$ over $[0, t_0]$, whereas in the latter representation of $H$, we integrate with respect to $y$ over $[t_0, 1]$. Since increments of the Brownian sheet are independent, $V$ and $H$ are independent. From the representation on the right-hand side, it follows analogously to $V \stackrel{\Dc}{=} \tilde{V}$, that
\begin{equation}\label{eq:denom}
	H(s) \stackrel{\Dc}{=} \sigma \sqrt{\int_{t_0}^1 \Big(\id(y \le t_1) - \tfrac{t_1 - t_0}{1 - t_0}\Big)^2 \diff y} B^{(2)}\bigg( \int_0^s \Big(\frac{s}{2} - x\Big)^2  \diff x \bigg)
\end{equation}
Since $\int_{t_0}^1 \big(\id(y \le t_1) - \frac{t_1 - t_0}{1 - t_0}\big)^2 \diff y = \frac{(1-t_1)(t_1-t_0)}{1 - t_0}$, combining $V \stackrel{\Dc}{=} \tilde{V}$ and \eqref{eq:denom}, yields 
\[ \frac{\sup_{s\in[0, 1]}|V_n(s)|}{\sup_{s\in[0, 1]}|H_n(s)|} \convw \sqrt{\frac{t_0(1 - t_0)}{(1-t_1)(t_1-t_0)}} \frac{\sup_{s \in [0, 1]} \Big| B^{(1)}\Big( \int_0^s \big(\frac{s}{2}-z\big)^2 \diff z \Big)\Big|}{\sup_{s \in [0, 1]} \Big|B^{(2)}\Big( \int_0^s \big(\frac{s}{2} - x\big)^2  \diff x \Big)\Big|}, \]
under $H_0$, whereas the nominator diverges to $\infty$ under $H_1$. Let $I = \sup_{s \in[0, 1]} \int_0^s \big(\tfrac{s}{2} - x\big)^2 \diff x$, then
\begin{equation*}
	\sup_{s \in [0, 1]} \bigg| B^{(i)}\bigg( \int_0^s \Big(\frac{s}{2}-z\Big)^2 \diff z \bigg)\bigg|
	= \sup_{v \in [0, I]}| B^{(i)}(v)|
	= \sup_{v \in [0, 1]}| B^{(i)}(Iv)| 
	\stackrel{\Dc}{=} \sqrt{I} \sup_{v \in [0, 1]}| B^{(i)}(v)|,
\end{equation*}
for $i=1, 2$. In particular, it follows that
\begin{equation*}
	\sqrt{\tfrac{t_0(1 - t_0)}{(1-t_1)(t_1-t_0)}} \frac{\sup_{s \in [0, 1]} \Big| B^{(1)}\Big( \int_0^s \big(\frac{s}{2}-z\big)^2 \sigma^2(z) \diff z \Big)\Big|}{\sup_{s \in [0, 1]} \Big|B^{(2)}\Big( \int_0^s \big(\frac{s}{2} - x\big)^2 \sigma^2(x) \diff x \Big)\Big|}
	\stackrel{\Dc}{=} \sqrt{\tfrac{t_0(1 - t_0)}{(1-t_1)(t_1-t_0)}} \frac{\sup_{v \in [0, 1]}| B^{(1)}(v)|}{\sup_{v \in [0, 1]}|B^{(2)}(v)|},
\end{equation*}
which finishes the proof.

\noindent \textbf{Proof of Corollary \ref{cor:local_alternatives}}.

1. \textit{(Local abrupt alternatives)} Note that $\mu(t)$ is piecewise Lipschitz continuous, such that 
$\sup_{s \in [0, 1]} |H_n(s)|$ converges weakly to $\sup_{s \in [0, 1]} |H(s)|$, with $H$ as in the proof of Corollary \ref{cor:full}. Moreover, by Proposition \ref{prop:mu_approx},	
\begin{equation*}
	\ex[\tilde{S}_n(t, s)] = \frac{\lfloor \tfrac{nt}{\ell_n}\rfloor - 1}{b_n} \big(s \mu_0 + (s - \tilde{t}) a_n \id(s > \tilde{t})\big) + \Oc\big(\tfrac{b_n}{n}\big),
\end{equation*}
uniformly for $s, t \in[0, 1]$. By definition, $\ex[V_n(s)] = \Oc\big(\tfrac{b_n}{\sqrt{n}}\big)$ for $s \le \tilde{t}$. By a straightforward calculation,  
\begin{align}\label{eq:abrupt_alt}
	\ex[V_n(s)]
	& = \sqrt{n} \frac{\lfloor \tfrac{nt_0}{\ell_n}\rfloor - 1}{b_n} \int_0^s \bigg(\int_0^x \mu_0 + a_n \id(z > \tilde{t}) \diff z - \frac{x}{s} \int_0^s \mu_0 + a_n \id(z > \tilde{t}) \diff z \bigg)\diff x + \Oc\big(\tfrac{b_n}{\sqrt{n}}\big) \\
	& = - \sqrt{n} \frac{\lfloor \tfrac{nt_0}{\ell_n}\rfloor - 1}{b_n} \frac{a_n}{2}(s-\tilde{t})\tilde{t} + \Oc\big(\tfrac{b_n}{\sqrt{n}}\big), \notag
\end{align}
for $s > \tilde{t}$. In particular, $\ex[V_n(s)]$ converges to $-\tfrac{t_0}{2} \max\{s-\tilde{t}, 0\}\tilde{t} d =: m_d(s)$, uniformly for $s\in[0, 1]$ as $n\to\infty$. Moreover, recall that $\{V_n(s) - \ex[V_n(s)]\}_{s\in[0, 1]}$ converges weakly to $V$, with $V$ as in the proof of Corollary \ref{cor:full}. If $d = \infty$,
\begin{equation*}
	\sup_{s\in [0, 1]} |V_n(s)|
	\ge \sup_{s\in [0, 1]} |\ex[V_n(s)]| - \sup_{s\in [0, 1]} |V_n(s) - \ex[V_n(s)]| \to \infty
\end{equation*}
by the triangle inequality. Conversely, if $d < \infty$ and $\sigma^2(x) \equiv \sigma^2 > 0$, the covariance structure of $V$ is non-degenerate, such that $m_d$ is in the support of the law of $V$. By the strict Anderson inequality \citep[see Corollary 2 of][]{lewandowski1995} and independence of $V$ and $H$,   
\begin{align} \notag
	& \pr\bigg( \frac{\sup_{s \in [0, 1]} |V_n(s)|}{\sup_{s \in [0, 1]} |H_n(s)|} > \sqrt{\tfrac{t_0(1-t_0)}{(1-t_1)(t_1 - t_0)}} q_{1-\alpha} \bigg) \\
	& \xrightarrow{n \to \infty} \pr\bigg( \frac{\sup_{s \in [0, 1]} |V(s) + m_d(s)|}{\sup_{s \in [0, 1]} |H(s)|} > \sqrt{\tfrac{t_0(1-t_0)}{(1-t_1)(t_1 - t_0)}} q_{1-\alpha} \bigg) \label{eq:anderson}\\
	& > \pr\bigg( \frac{\sup_{s \in [0, 1]} |V(s)|}{\sup_{s \in [0, 1]} |H(s)|} > \sqrt{\tfrac{t_0(1-t_0)}{(1-t_1)(t_1 - t_0)}} q_{1-\alpha} \bigg) = \alpha. \notag
\end{align}

\noindent 2. \textit{(Local smooth alternatives)} For $s < \tilde{t}$, it holds $s < \tilde{t} - c_n$ for almost every $n\in\N$. In this case, $\ex[V_n(s)] = \Oc\big(\tfrac{b_n}{\sqrt{n}}\big)$, by the same arguments as before. Similarly, for $s > \tilde{t}$, $s > \tilde{t} + c_n$ for almost every $n\in\N$. By Proposition \ref{prop:mu_approx},
\begin{align*}
	\ex[V_n(s)] 
	& = \sqrt{n} \frac{\lfloor \tfrac{nt_0}{\ell_n}\rfloor - 1}{b_n} \int_0^s \bigg(\int_0^x \mu(z) \diff z - \frac{x}{s} \int_0^s \mu(z) \diff z \bigg)\diff x + \Oc\big(\tfrac{b_n}{\sqrt{n}}\big) \\
	& = \sqrt{n} \frac{\lfloor \tfrac{nt_0}{\ell_n}\rfloor - 1}{b_n} a_n c_n \Big(\frac{s}{2} - \tilde{t}\Big) + \Oc\big(\tfrac{b_n}{\sqrt{n}}\big),
\end{align*}
since $\mu(x) = \mu_0 + a_n h(\tfrac{x - \tilde{t}}{c_n})$ and $h$ has support $[-1, 1]$ with $\int h(x) \diff x = 1$ and $\int x h(x) \diff x = 0$. Similarly, we obtain \begin{equation*}
	\ex[V_n(\tilde{t})] = \sqrt{n} \frac{\lfloor \tfrac{nt_0}{\ell_n}\rfloor - 1}{b_n} a_n c_n \Big(-\frac{\tilde{t}}{4} - c_n \int_{-1}^0 x h(x) \diff x\Big) + \Oc\big(\tfrac{b_n}{\sqrt{n}}\big).
\end{equation*}
As before, $\sup_{s\in [0, 1]} |V_n(s) - \ex[V_n(s)]|$ converges weakly to $\sup_{s\in[0, 1]} |V(s)|$, such that the asymptotic behavior of $\sup_{s\in[0, 1]} |V_n(s)|$ is controlled by $\sup_{s\in[0, 1]} |\ex[V_n(s)]|$. In particular,
\begin{equation*}
	\sup_{s\in [0, 1]} |\ex[V_n(s)]|
	= \sqrt{n} \frac{\lfloor \tfrac{nt_0}{\ell_n}\rfloor - 1}{b_n} a_n c_n \max\{|\tfrac{1}{2} - \tilde{t}|, \tfrac{\tilde{t}}{4}\} + \Oc\big(\tfrac{b_n}{\sqrt{n}}\big) + \Oc(\sqrt{n}a_n c_n^2),
\end{equation*}
which diverges to $\infty$, whenever $\lim_{n\to\infty} \sqrt{n}a_nc_n = \infty$. 

Finally, let $d < \infty$ and $\sigma^2(x) \equiv \sigma^2 > 0$, such that the covariance structure of $V$ is non-degenerate. Note that the limit of $\ex[V_n(s)]$ is not continuous, and more effort is needed than in the case of abrupt alternatives. Let $m_d(s) = t_0 (s/2 - \tilde{t}) d \cdot  \id(s \in [\tilde{t}, 1])$. Then $m_d(s)$ is continuous on $[\tilde{t}, 1]$. Since $V(s)$ is continuous on $[0, t^*]$ and $V(s) + m_d(s)$ is continuous on $[t^*, 1]$, it holds
\begin{align*}
	\sup_{s\in[0, 1]} |V(s) + m_d(s)| 
	& = \max\{\sup_{s\in[0, t^*)} |V(s) + m_d(s)|, |V(t^*) + m_d(t^*)|,  \sup_{s\in[t^*, 1]} |V(s) + m_d(s)|\} \\
	& \ge \max\{\sup_{s\in[0, t^*)} |V(s) + m_d(s)|,   \sup_{s\in(t^*, 1]} |V(s) + m_d(s)|\} \\ 
	& = \max\{\sup_{s\in[0, t^*]} |V(s)|,   \sup_{s\in[t^*, 1]} |V(s) + \tilde m_d(s)|\}.
\end{align*}
Now, by considering the product space $C([0, t^*]) \times C([t^*, 1])$ and the same arguments as for local abrupt alternatives, we have by the strict Anderson inequality, analogously to \eqref{eq:anderson},
\begin{equation*} 
	\lim_{n\to\infty} \pr\Big( \tfrac{\sup_{s \in [0, 1]} |V_n(s)|}{\sup_{s \in [0, 1]} |H_n(s)|} > \sqrt{\tfrac{t_0(1-t_0)}{(1-t_1)(t_1 - t_0)}} q_{1-\alpha} \Big) 
	> \pr\Big( \tfrac{\sup_{s \in [0, 1]} |V(s)|}{\sup_{s \in [0, 1]} |H(s)|} > \sqrt{\tfrac{t_0(1-t_0)}{(1-t_1)(t_1 - t_0)}} q_{1-\alpha} \Big) = \alpha. 
\end{equation*}

\noindent \textbf{Proof of Corollary \ref{cor:estimation}}.

First consider the case $s^* < \infty$. By Proposition \ref{prop:mu_approx}, 
\begin{align} \notag
	\ex[V_n(s)] 
	& = \sqrt{n} \frac{\lfloor \tfrac{nt_0}{\ell_n}\rfloor - 1}{b_n} \int_0^s \bigg(\int_0^x \mu(z) \diff z - \frac{x}{s} \int_0^s \mu(z) \diff z \bigg)\diff x + \Oc\big(\tfrac{b_n}{\sqrt{n}}\big) \\
	& = \sqrt{n} \frac{\lfloor \tfrac{nt_0}{\ell_n}\rfloor - 1}{b_n} \bigg(\int_0^s \int_y^s \diff x \mu(y) \diff y + \frac{s}{2}\int_0^s \mu(y) \diff y\bigg) + \Oc\big(\tfrac{b_n}{\sqrt{n}}\big) \label{eq:estimator_decomposition} \\
	& = \sqrt{n} \frac{\lfloor \tfrac{nt_0}{\ell_n}\rfloor - 1}{b_n} \int_0^s \Big(\frac{s}{2}-y\Big)\mu(y) \diff y + \Oc\big(\tfrac{b_n}{\sqrt{n}}\big).\notag
\end{align}
For $s < s^*$, $\mu$ is constant, so that
\begin{equation} \label{eq:estimator_decomposition2}
	\int_0^s \Big(\frac{s}{2}-y\Big)\mu(y) \diff y = \mu(0) \big[\frac{s^2}{2} - \frac{s^2}{2}\big] = 0
\end{equation}
for $s \le s^*$. Therefore,
\begin{equation} \label{eq:estimator_lower_bound}
	\pr(\hat{s}^* < s^*) 
	\le \pr(\sup_{s \in [0, s^*]} |V_n(s)| > c_n)
	= \pr(\sup_{s \in [0, s^*]} |V_n(s) - \ex[V_n(s)] +  \Oc\big(\tfrac{b_n}{\sqrt{n}}\big)| > c_n),
\end{equation}
which vanishes as $n \to \infty$, since $c_n\to \infty$, $b_n^2/n \to 0$ by Assumption \ref{assump:sequences} and $\sup_{s \in [0, s^*]} |V_n(s) - \ex[V_n(s)] \convw \sup_{s \in [0, s^*]} |V(s)|$, with $V$ as in the proof of Corollary \ref{cor:full}.

Now, let $\delta > 0$ such that $\mu$ is Lipschitz continuous in $(s^*, s^* + \delta)$. Then, 
\begin{align*}
	\int_0^{s^* + \delta} \Big(\frac{s^* + \delta}{2}-y\Big)\mu(y) \diff y
	& = \mu(0) \int_0^{s^*} \Big(\frac{s^* + \delta}{2}-y\Big) \diff y + \int_{s^*}^{s^* + \delta} \Big(\frac{s^* + \delta}{2}-y\Big)\mu(y) \diff y \\
	& = \mu(0) \delta \frac{s^*}{2} + \int_0^\delta \Big(\frac{-s^* + \delta}{2}-y\Big)\mu(s^* + y) \diff y.
\end{align*}
By assumption, $\mu(s^* + y) = \mu(0) + c_\kappa y^\kappa + o(\delta^\kappa)$, uniformly for $y \in (s^*, s^* + \delta)$.  Hence, 
\begin{equation*}
	 \int_0^\delta \Big(\frac{-s^* + \delta}{2}-y\Big)\mu(s^* + y) \diff y
	 = - \mu(0) \delta \frac{s^*}{2} - \frac{c_\kappa \delta^{\kappa+1}}{2(\kappa + 1)} \Big(s^* + \frac{\delta \kappa}{\kappa + 2}\Big) + o(\delta^{\kappa + 1}).
\end{equation*}
In particular,
\begin{equation} \label{eq:estimator_integral}
	\int_0^{s^* + \delta} \Big(\frac{s^* + \delta}{2}-y\Big)\mu(y) \diff y =  \frac{c_\kappa \delta^{\kappa+1}}{2(\kappa + 1)}\Big(s^* + \frac{\delta \kappa}{\kappa + 2}\Big) + o(\delta^{\kappa+1}).
\end{equation}
Let $\delta_n = M \big(\tfrac{c_n}{\sqrt{n}}\big)^{1/(\kappa + 1)}$, for some constant $M \ge 0$, such that $\delta_n < \delta$. Analogously to \eqref{eq:estimator_lower_bound}, 
\begin{align} \label{eq:estimator_decomposition3}
	\pr(\hat{s}^* > s^* +  \delta_n) 
	& \le \pr(\sup_{s \in [0, s^* + \delta_n ]} |V_n(s)| \le c_n) \\
	& \le \pr(\sup_{s \in [0, s^* + \delta_n]} |\ex[V_n(s)]| - \sup_{s \in [0, s^* + \delta_n]} |V_n(s) - \ex[V_n(s)]| \le c_n), \notag
\end{align}
by the triangle inequality. First note, that $\sup_{s \in [0, s^* + \delta_n]} |V_n(s) - \ex[V_n(s)]| = \Oc_\pr(1)$, since $\sup_{s \in [0, 1]} |V_n(s) - \ex[V_n(s)]| \convw \sup_{s \in [0, 1]} |V(s)|$. Combing \eqref{eq:estimator_decomposition} and \eqref{eq:estimator_integral}, we obtain
\begin{align*}
	\sup_{s \in [0, s^* + \delta_n]} |\ex[V_n(s)]| 
	\ge |\ex[V_n(s^* + \delta_n)]|
	& = c_n \bigg(\frac{\lfloor \tfrac{nt_0}{\ell_n}\rfloor - 1}{b_n} \frac{|c_\kappa| M^{\kappa+1}}{2(\kappa + 1)} s^* + o(1)\bigg) + \Oc\big(\tfrac{b_n}{\sqrt{n}}\big).
\end{align*}
By choosing $M$ sufficiently large, $\sup_{s \in [0, s^* + \delta_n]} |\ex[V_n(s)]| \ge 2 c_n$, such that $\lim_{n\to\infty}\pr(\hat{s}^* > s^* +  \delta_n)  = 0$ by \eqref{eq:estimator_decomposition3}.

If $s^* = \infty$, $\ex[V_n(s)] = 0$ by \eqref{eq:estimator_decomposition2}, such that 
\begin{align*}
	\pr(\hat{s}^* < \infty)
	= \pr(\sup_{s \in [0,1]} |V_n(s)| > c_n)
	= \pr(\sup_{s \in [0,1]} |V_n(s) - \ex[V_n(s)]| > c_n),
\end{align*}
which converges to $0$ since $\sup_{s \in [0,1]} |V_n(s) - \ex[V_n(s)]|\convw \sup_{s \in [0,1]} |V(s)|$ and $c_n \to \infty$.

\section*{Acknowledgements}

The author thanks Fabian Mies for carefully reading an earlier version of this manuscript and for pointing out a critical error in the proof of a previous result, which led to substantial improvements in the present version.
	
\bibliography{bibliography}

\newpage

\appendix


\section{Additional Empirical Results} \label{app:appendix_empirical_results}

\begin{table}[H]
	\caption{Empirical rejection rates for various choices of $\eps$ and $\sigma$ under the null hypothesis $\mu = \mu_0$.} \label{tab:eps0}
	\begin{tabular}{ll|rrrrrrr}
		\toprule
		$\eps$ & $\sigma$ & R1 & R2 & SN & BT & LRV & \eqref{eq:decision_rule0} & \eqref{eq:decision_rule} \\
		\midrule
		\multicolumn{9}{l}{\textit{Panel A: $n = 200$}}\\		
		iid & $\sigma_3$ & 37.80 & 61.60 & 3.10 & 24.90 & 9.40 & 5.60 & 0.50 \\
		ar & $\sigma_3$ & 98.20 & 99.70 & 0.30 & 33.40 & 30.50 & 28.10 & 2.70 \\
		ma & $\sigma_3$ & 78.40 & 88.60 & 5.00 & 24.40 & 13.80 & 9.70 & 0.40 \\
		ls & $\sigma_0$ & 77.80 & 93.20 & 4.40 & 33.50 & 22.30 & 11.20 & 3.10 \\
		ls & $\sigma_1$ & 87.50 & 94.00 & 0.00 & 25.10 & 31.70 & 16.30 & 1.00 \\
		ls & $\sigma_2$ & 81.40 & 95.20 & 3.10 & 51.80 & 20.80 & 14.50 & 2.80 \\
		ls & $\sigma_3$ & 88.80 & 93.50 & 0.20 & 25.80 & 34.80 & 18.40 & 1.80 \\
		\midrule
		\multicolumn{9}{l}{\textit{Panel B: $n = 500$}}\\		
		iid & $\sigma_3$ & 37.80 & 64.50 & 4.20 & 22.90 & 9.00 & 4.80 & 2.70 \\
		ar & $\sigma_3$ & 99.60 & 99.90 & 0.10 & 27.10 & 21.70 & 22.90 & 2.90 \\
		ma & $\sigma_3$ & 75.80 & 87.00 & 2.10 & 22.80 & 13.60 & 11.10 & 2.20 \\
		ls & $\sigma_0$ & 77.90 & 88.70 & 4.20 & 29.50 & 19.80 & 11.70 & 4.10 \\
		ls & $\sigma_1$ & 93.30 & 97.60 & 1.60 & 17.10 & 26.80 & 16.20 & 4.50 \\
		ls & $\sigma_2$ & 78.00 & 92.00 & 5.70 & 47.20 & 17.40 & 13.90 & 2.30 \\
		ls & $\sigma_3$ & 96.70 & 98.40 & 1.00 & 20.00 & 28.00 & 19.90 & 3.50 \\
		\midrule
		\multicolumn{9}{l}{\textit{Panel C: $n = 1000$}}\\		
		iid & $\sigma_3$ & 41.40 & 67.40 & 6.30 & 19.30 & 7.90 & 3.30 & 2.70 \\
		ar & $\sigma_3$ & 99.90 & 100.00 & 0.00 & 25.70 & 18.00 & 18.90 & 4.90 \\
		ma & $\sigma_3$ & 84.00 & 91.10 & 3.40 & 23.90 & 13.70 & 10.40 & 3.00 \\
		ls & $\sigma_0$ & 80.80 & 90.00 & 8.00 & 26.20 & 18.90 & 12.90 & 3.80 \\
		ls & $\sigma_1$ & 94.20 & 97.30 & 2.10 & 15.30 & 25.00 & 14.50 & 3.40 \\
		ls & $\sigma_2$ & 80.80 & 91.50 & 8.00 & 47.40 & 16.20 & 13.40 & 2.40 \\
		ls & $\sigma_3$ & 98.60 & 99.00 & 0.40 & 18.90 & 25.60 & 16.90 & 3.30 \\
		\bottomrule
	\end{tabular}
\end{table}

\begin{table}[H]
	\caption{Empirical rejection rates for various choices of $\eps$ and $\sigma$ under the alternative $\mu = \mu_5$.} \label{tab:eps1}
	\begin{tabular}{ll|rrrrrrr}
		\toprule
		$\eps$ & $\sigma$ & R1 & R2 & SN & BT & LRV & \eqref{eq:decision_rule0} & \eqref{eq:decision_rule} \\
		\midrule
		\multicolumn{9}{l}{\textit{Panel A: $n = 200$}}\\	
		iid & $\sigma_3$ & 100.00 & 100.00 & 76.00 & 100.00 & 100.00 & 100.00 & 97.00 \\
		ar & $\sigma_3$ & 100.00 & 100.00 & 4.30 & 100.00 & 100.00 & 100.00 & 98.20 \\
		ma & $\sigma_3$ & 100.00 & 100.00 & 40.10 & 100.00 & 100.00 & 100.00 & 94.50 \\
		ls & $\sigma_0$ & 100.00 & 100.00 & 76.00 & 100.00 & 100.00 & 100.00 & 99.40 \\
		ls & $\sigma_1$ & 100.00 & 100.00 & 40.50 & 100.00 & 100.00 & 100.00 & 98.10 \\
		ls & $\sigma_2$ & 100.00 & 100.00 & 85.90 & 100.00 & 100.00 & 100.00 & 99.90 \\
		ls & $\sigma_3$ & 100.00 & 100.00 & 27.20 & 100.00 & 100.00 & 100.00 & 98.30 \\
		\midrule
		\multicolumn{9}{l}{\textit{Panel B: $n = 500$}}\\	
		iid & $\sigma_3$ & 100.00 & 100.00 & 99.10 & 100.00 & 100.00 & 100.00 & 100.00 \\
		ar & $\sigma_3$ & 100.00 & 100.00 & 3.70 & 100.00 & 100.00 & 100.00 & 100.00 \\
		ma & $\sigma_3$ & 100.00 & 100.00 & 39.90 & 100.00 & 100.00 & 100.00 & 100.00 \\
		ls & $\sigma_0$ & 100.00 & 100.00 & 95.40 & 100.00 & 100.00 & 100.00 & 100.00 \\
		ls & $\sigma_1$ & 100.00 & 100.00 & 54.00 & 100.00 & 100.00 & 100.00 & 100.00 \\
		ls & $\sigma_2$ & 100.00 & 100.00 & 95.30 & 100.00 & 100.00 & 100.00 & 100.00 \\
		ls & $\sigma_3$ & 100.00 & 100.00 & 27.40 & 100.00 & 100.00 & 100.00 & 100.00 \\
		\midrule
		\multicolumn{9}{l}{\textit{Panel C: $n = 1000$}}\\	
		iid & $\sigma_3$ & 100.00 & 100.00 & 100.00 & 100.00 & 100.00 & 100.00 & 100.00 \\
		ar & $\sigma_3$ & 100.00 & 100.00 & 0.60 & 100.00 & 100.00 & 100.00 & 100.00 \\
		ma & $\sigma_3$ & 100.00 & 100.00 & 56.70 & 100.00 & 100.00 & 100.00 & 100.00 \\
		ls & $\sigma_0$ & 100.00 & 100.00 & 98.80 & 100.00 & 100.00 & 100.00 & 100.00 \\
		ls & $\sigma_1$ & 100.00 & 100.00 & 58.20 & 100.00 & 100.00 & 100.00 & 100.00 \\
		ls & $\sigma_2$ & 100.00 & 100.00 & 99.90 & 100.00 & 100.00 & 100.00 & 100.00 \\
		ls & $\sigma_3$ & 100.00 & 100.00 & 25.40 & 100.00 & 100.00 & 100.00 & 100.00 \\
		\bottomrule
	\end{tabular}
\end{table}

\begin{table}[H]
	\caption{Empirical rejection rates for various choices of $\mu$ for $\sigma=\sigma_3$ and (ls) errors.} \label{tab:mu}
	\begin{tabular}{ll|rrrrrrr}
		\toprule
		n & $\mu$ & R1 & R2 & SN & BT & LRV & \eqref{eq:decision_rule0} & \eqref{eq:decision_rule} \\
		\midrule
		200 & $\mu_0$ & 88.80 & 93.50 & 0.20 & 25.80 & 34.80 & 18.40 & 1.80 \\
		500 & $\mu_0$ & 96.70 & 98.40 & 1.00 & 20.00 & 28.00 & 19.90 & 3.50 \\
		1000 & $\mu_0$ & 98.60 & 99.00 & 0.40 & 18.90 & 25.60 & 16.90 & 3.30 \\
		\midrule
		200 & $\mu_1$ & 100.00 & 100.00 & 1.20 & 98.70 & 99.60 & 94.90 & 9.90 \\
		200 & $\mu_2$ & 100.00 & 100.00 & 31.10 & 100.00 & 100.00 & 88.20 & 99.90 \\
		200 & $\mu_3$ & 99.90 & 100.00 & 3.80 & 100.00 & 100.00 & 100.00 & 65.20 \\
		200 & $\mu_4$ & 99.90 & 100.00 & 1.60 & 100.00 & 99.60 & 98.70 & 37.40 \\
		200 & $\mu_5$ & 100.00 & 100.00 & 27.20 & 100.00 & 100.00 & 100.00 & 98.30 \\
		200 & $\mu_6$ & 100.00 & 100.00 & 4.00 & 100.00 & 100.00 & 100.00 & 65.50 \\
		500 & $\mu_1$ & 100.00 & 100.00 & 7.20 & 100.00 & 100.00 & 100.00 & 52.30 \\
		500 & $\mu_2$ & 100.00 & 100.00 & 31.30 & 100.00 & 100.00 & 100.00 & 100.00 \\
		500 & $\mu_3$ & 100.00 & 100.00 & 10.10 & 100.00 & 100.00 & 100.00 & 73.90 \\
		500 & $\mu_4$ & 100.00 & 100.00 & 9.20 & 100.00 & 100.00 & 100.00 & 84.40 \\
		500 & $\mu_5$ & 100.00 & 100.00 & 27.40 & 100.00 & 100.00 & 100.00 & 100.00 \\
		500 & $\mu_6$ & 100.00 & 100.00 & 9.10 & 100.00 & 100.00 & 100.00 & 95.60 \\
		1000 & $\mu_1$ & 100.00 & 100.00 & 9.80 & 100.00 & 100.00 & 100.00 & 86.90 \\
		1000 & $\mu_2$ & 100.00 & 100.00 & 28.50 & 100.00 & 100.00 & 100.00 & 100.00 \\
		1000 & $\mu_3$ & 100.00 & 100.00 & 7.70 & 100.00 & 100.00 & 100.00 & 99.80 \\
		1000 & $\mu_4$ & 100.00 & 100.00 & 7.60 & 100.00 & 100.00 & 100.00 & 93.40 \\
		1000 & $\mu_5$ & 100.00 & 100.00 & 25.40 & 100.00 & 100.00 & 100.00 & 100.00 \\
		1000 & $\mu_6$ & 100.00 & 100.00 & 7.20 & 100.00 & 100.00 & 100.00 & 99.70 \\
		\bottomrule
	\end{tabular}
\end{table}

\setlength{\tabcolsep}{3pt}
\begin{table}[H] \footnotesize
	\caption{Computation time for each iteration in ms.}  \label{tab:computation_time}
	\begin{tabular}{l|rrrrrr} 
	\toprule
	& R1 & SN & BT & LRV & \eqref{eq:decision_rule0} & \eqref{eq:decision_rule} \\
	n &  &  &  &  &  &  \\
	\midrule
	200 & 0.477 ($\pm$ 0.025) & 0.970 ($\pm$ 0.021) & 1.445 ($\pm$ 0.018) & 0.154 ($\pm$ 0.006) & 0.251 ($\pm$ 0.004) & 0.264 ($\pm$ 0.004) \\
	500 & 0.635 ($\pm$ 0.053) & 1.167 ($\pm$ 0.075) & 3.379 ($\pm$ 0.081) & 0.165 ($\pm$ 0.006) & 1.410 ($\pm$ 0.007) & 1.425 ($\pm$ 0.006) \\
	1000 & 0.883 ($\pm$ 0.084) & 1.600 ($\pm$ 0.244) & 7.328 ($\pm$ 0.187) & 0.181 ($\pm$ 0.007) & 5.552 ($\pm$ 0.023) & 5.570 ($\pm$ 0.018) \\
	\bottomrule
\end{tabular}
\end{table}

\begin{table}
	\caption{Empirical rejection rates under local abrupt alternatives.} \label{tab:local_abrupt}
	\begin{tabular}{l|rrrrrrr}
		\toprule
		height & R1 & R2 & SN & BT & LRV & \eqref{eq:decision_rule0} & \eqref{eq:decision_rule} \\
		\midrule	
		-32 & 0.00 & 0.50 & 0.00 & 100.00 & 100.00 & 100.00 & 100.00 \\
		-16 & 1.10 & 100.00 & 0.20 & 100.00 & 100.00 & 100.00 & 100.00 \\
		-8 & 99.90 & 100.00 & 3.60 & 100.00 & 100.00 & 100.00 & 100.00 \\
		-4 & 100.00 & 100.00 & 7.50 & 100.00 & 100.00 & 100.00 & 100.00 \\
		-2 & 100.00 & 100.00 & 12.60 & 100.00 & 100.00 & 100.00 & 100.00 \\
		-1 & 100.00 & 100.00 & 10.40 & 100.00 & 100.00 & 100.00 & 96.70 \\
		-0.5 & 100.00 & 100.00 & 7.00 & 100.00 & 100.00 & 100.00 & 76.90 \\
		-0.25 & 98.60 & 99.70 & 2.70 & 100.00 & 91.80 & 99.10 & 44.40 \\
		-0.125 & 96.40 & 98.00 & 1.10 & 99.80 & 55.80 & 73.80 & 16.50 \\
		-0.0625 & 96.60 & 99.20 & 0.80 & 64.00 & 35.80 & 38.40 & 7.00 \\
		-0.03125 & 96.50 & 98.10 & 0.60 & 30.20 & 30.50 & 23.40 & 5.90 \\
		\midrule
		0 & 95.00 & 97.60 & 1.00 & 21.00 & 28.30 & 23.00 & 3.10 \\
		\midrule
		0.03125 & 93.80 & 96.70 & 0.90 & 31.50 & 31.40 & 27.80 & 4.20 \\
		0.0625 & 96.30 & 98.00 & 0.80 & 62.90 & 34.80 & 40.90 & 6.20 \\
		0.125 & 95.70 & 98.50 & 1.10 & 99.60 & 53.20 & 72.20 & 16.70 \\
		0.25 & 98.20 & 99.70 & 3.30 & 100.00 & 91.30 & 98.70 & 43.20 \\
		0.5 & 99.80 & 100.00 & 6.70 & 100.00 & 100.00 & 100.00 & 78.20 \\
		1 & 100.00 & 100.00 & 7.10 & 100.00 & 100.00 & 100.00 & 97.60 \\
		2 & 100.00 & 100.00 & 12.10 & 100.00 & 100.00 & 100.00 & 100.00 \\
		4 & 100.00 & 100.00 & 18.10 & 100.00 & 100.00 & 100.00 & 100.00 \\
		8 & 99.40 & 100.00 & 3.70 & 100.00 & 100.00 & 100.00 & 100.00 \\
		16 & 2.30 & 100.00 & 82.90 & 100.00 & 100.00 & 100.00 & 100.00 \\
		32 & 0.00 & 1.30 & 0.00 & 100.00 & 100.00 & 100.00 & 100.00 \\
		\bottomrule
	\end{tabular}
\end{table}

\begin{table}
	\caption{Empirical rejection rates under local smooth alternatives.} \label{tab:local_bump}
	\begin{tabular}{l|rrrrrrr}
		\toprule
		height & R1 & R2 & SN & BT & LRV & \eqref{eq:decision_rule0} & \eqref{eq:decision_rule} \\
		\midrule
		-32 & 63.30 & 100.00 & 100.00 & 100.00 & 100.00 & 100.00 & 100.00 \\
		-16 & 100.00 & 100.00 & 100.00 & 100.00 & 100.00 & 100.00 & 100.00 \\
		-8 & 100.00 & 100.00 & 38.90 & 100.00 & 100.00 & 100.00 & 100.00 \\
		-4 & 100.00 & 100.00 & 19.70 & 100.00 & 100.00 & 100.00 & 96.00 \\
		-2 & 100.00 & 100.00 & 12.60 & 100.00 & 100.00 & 100.00 & 70.30 \\
		-1 & 100.00 & 100.00 & 7.80 & 100.00 & 95.70 & 93.70 & 32.10 \\
		-0.5 & 99.70 & 100.00 & 2.40 & 98.30 & 43.60 & 62.50 & 10.90 \\
		-0.25 & 96.90 & 98.70 & 1.50 & 50.40 & 32.80 & 33.10 & 6.00 \\
		-0.125 & 95.90 & 98.10 & 0.40 & 22.80 & 32.80 & 20.50 & 4.70 \\
		-0.0625 & 95.00 & 98.10 & 0.90 & 21.10 & 26.90 & 24.60 & 3.30 \\
		-0.03125 & 97.20 & 98.80 & 0.40 & 21.10 & 29.10 & 21.70 & 3.30 \\
		\midrule
		0 & 95.60 & 97.50 & 0.60 & 20.80 & 28.00 & 22.30 & 3.50 \\
		\midrule
		0.03125 & 95.80 & 97.50 & 1.20 & 18.00 & 27.90 & 20.90 & 3.60 \\
		0.0625 & 97.40 & 98.50 & 0.70 & 21.40 & 26.70 & 21.60 & 3.60 \\
		0.125 & 94.70 & 97.60 & 0.70 & 30.90 & 29.20 & 27.30 & 5.10 \\
		0.25 & 97.70 & 99.20 & 0.90 & 51.40 & 30.90 & 33.60 & 5.20 \\
		0.5 & 99.80 & 100.00 & 1.60 & 98.70 & 46.60 & 60.10 & 13.10 \\
		1 & 100.00 & 100.00 & 3.80 & 100.00 & 95.00 & 94.00 & 31.50 \\
		2 & 100.00 & 100.00 & 10.30 & 100.00 & 100.00 & 100.00 & 73.30 \\
		4 & 100.00 & 100.00 & 20.30 & 100.00 & 100.00 & 100.00 & 96.40 \\
		8 & 100.00 & 100.00 & 43.30 & 100.00 & 100.00 & 100.00 & 100.00 \\
		16 & 100.00 & 100.00 & 100.00 & 100.00 & 100.00 & 100.00 & 100.00 \\
		32 & 60.60 & 100.00 & 100.00 & 100.00 & 100.00 & 100.00 & 100.00 \\
		\bottomrule
	\end{tabular}
\end{table}

\end{document}